\documentclass[12pt,reqno]{amsart}
\usepackage{amsmath}
\usepackage{amssymb}
\usepackage{amsfonts}
\usepackage{amsthm}
\usepackage{mathrsfs}
\usepackage[all]{xy}
\usepackage{color}
\usepackage{multicol}
\usepackage{enumerate}
\usepackage{longtable}
\usepackage{array}

\usepackage[vcentermath,enableskew]{youngtab}
\usepackage{multirow}
\usepackage{tikz}
\usetikzlibrary{arrows}

\usepackage[colorlinks=true, pdfstartview=FitV, linkcolor=blue, citecolor=blue, urlcolor=blue]{hyperref}

\theoremstyle{plain}
\newtheorem*{theorem*}{Theorem}
\newtheorem{theorem}{Theorem}[section]
\newtheorem{lemma}[theorem]{Lemma}

\newtheorem{corollary}[theorem]{Corollary}
\newtheorem{proposition}[theorem]{Proposition}

\theoremstyle{definition}
\newtheorem{definition}[theorem]{Definition}
\newtheorem{remark}[theorem]{Remark}

\newtheorem{example}[theorem]{Example}

\newcommand{\C}{\mathbb{C}}
\newcommand{\Z}{\mathbb{ Z}}

\newcommand{\mO}{\mathbb{O}}

\renewcommand{\P}{\mathbb{ P}}

\newcommand{\la}{\lambda}

\newcommand{\cB}{\mathcal{B}}
\newcommand{\cC}{\mathcal{C}}
\newcommand{\cF}{\mathcal{F}}
\newcommand{\cN}{\mathcal{N}}
\newcommand{\cO}{\mathcal{O}}

\newcommand{\cQ}{\mathcal{Q}}
\newcommand{\cS}{\mathcal{S}}

\DeclareMathOperator{\End}{End}

\DeclareMathOperator{\Gl}{GL}

\DeclareMathOperator{\im}{im}

\DeclareMathOperator{\Type}{Type}
\DeclareMathOperator{\eType}{eType}

\DeclareMathOperator{\Sp}{Sp}
\DeclareMathOperator{\Lie}{Lie}

\newcommand{\gl}{\mathfrak{gl}}

\newcommand{\fN}{\mathfrak{N}}
\newcommand{\fZ}{\mathfrak{Z}}

\newcommand{\spp}{\mathfrak{sp}}

\newcommand{\SYB}{SYB}
\newcommand{\ovm}{\overline{m}}

\begin{document}
\title{Irreducible Components of Exotic Springer Fibres}

\author{Vinoth Nandakumar}
\address{V.~Nandakumar: School of Mathematics and Statistics, University of Sydney}
\email{vinoth.nandakumar@sydney.edu.au}

\author{Daniele Rosso}
\address{D.~Rosso: Department of Mathematics and Actuarial Science, Indiana University Northwest}
\email{drosso@iu.edu}

\author{Neil Saunders}
\address{N.~Saunders: Department of Mathematical Sciences, Old Royal Naval College,  University of Greenwich  and Honorary Associate at the School of Mathematics and Statistics, University of Sydney}
\email{n.saunders@greenwich.ac.uk}

\maketitle

\begin{abstract}
Kato introduced the exotic nilpotent cone to be a substitute for the ordinary nilpotent cone of type C with cleaner properties. Here we describe the irreducible components of exotic Springer fibres (the fibres of the resolution of the exotic nilpotent cone), and prove that they are naturally in bijection with standard bitableaux. As a result, we deduce the existence of an exotic Robinson-Schensted bijection, which is a variant of the type C Robinson-Schensted bijection between pairs of same-shape standard bitableaux and elements of the Weyl group; this bijection is described explicitly in the sequel to this paper. Note that this is in contrast with ordinary type C Springer fibres, where the parametrisation of irreducible components, and the resulting geometric Robinson-Schensted bijection, are more complicated. As an application, we explicitly describe the structure in the special cases where the irreducible components of the exotic Springer fibre have dimension 2, and show that in those cases one obtains Hirzebruch surfaces.\end{abstract}
\setcounter{tocdepth}{1}
\tableofcontents

\section{Introduction}
The Springer Correspondence gives a bijection between the irreducible representations of the Weyl group of a connected reductive algebraic group $G$, and certain pairs $(\cO, \varepsilon)$ comprising a $G$-orbit on the nilpotent cone $\mathcal{N} = \mathcal{N}(\mathfrak{g})$ of its Lie algebra $\mathfrak{g}$, and certain simple local systems $\varepsilon$ on $\cO$. In type $A$, the nilpotent cone $\cN(\gl_n)$ consists of all nilpotent $n \times n$ matrices, and the $G$-action (by conjugation) on $\mathcal{N}$ has connected stabilisers, so no non-trivial local systems occur in the Springer Correspondence. In this case, one has a bijection between the $G$-orbits on $\cN(\gl_n)$ and the irreducible representations of $S_n$, the Weyl group of $GL_n$, which are parametrised by partitions of $n$. 
\vspace{5pt}

The Springer Correspondence in type $C$, using the ordinary nilpotent cone, is more complicated than that in type $A$ for a number of reasons - the isotropy groups of the orbits are not connected, and one does not obtain a bijection between nilpotent orbits and irreducible representations of the Weyl group of type $C$. Kato introduced the exotic nilpotent cone $\mathfrak{N}$ to evade these complications, and Kato's exotic Springer correspondence, as constructed in \cite{kato2}, does give such a bijection. This relies on the fact that the $Sp_{2n}(\mathbb{C})$-orbits on $\mathfrak{N}$ are proven to be in bijection with $\mathcal{Q}_n$, the set of bipartitions of $n$ (and thus also with irreducible representations of the type $C$ Weyl group). \vspace{5pt} 

Exotic Springer fibres have also been used by Kato for various  geometric constructions. In \cite{kat}, Kato studies representations of multi-parameter affine Hecke algebras, by using the exotic nilpotent cone (and the equivariant K-theory of its Steinberg variety, following techniques used by Kazhdan, Lusztig and Ginzburg in the case of one-parameter affine Hecke algebras). In particular, the standard modules for these Hecke algebras are realised via the total homology of exotic Springer fibres and, as is the case for the classical Springer Correspondence, the top homology gives the irreducible representation of the Weyl group). In Corollary 1.24 of \cite{katd} the authors establish a connection between the homology of exotic Springer fibres, and of ordinary Springer fibres in types $B$ and $C$ (see also \cite{kato4}, for a purity result).  \vspace{5pt}

Following Kato's foundational papers on the exotic nilpotent cone, there has been subsequent work extending various results about the nilpotent cone to the exotic setting. In \cite{AH}, Achar and Henderson conjecturally describe the intersection cohomology of orbit closures in the exotic nilpotent cone; these have since been proven independently by Shoji-Sorlin (see Theorem $5.7$ in \cite{ss}); and by Kato (see Theorem A in \cite{kato3}, Theorem A.$1.8$ in \cite{kato4} and Remark $5.8$ in \cite{ss}). Achar, Henderson and Sommers make an explicit connection between special pieces for $\mathfrak{N}$ and those for the ordinary nilpotent cone in \cite{special}. The Lusztig-Vogan bijection can also be extended to the exotic nilpotent cone, as shown by the first author in \cite{vn}. These results all demonstrate a strong connection between the exotic nilpotent cone and the ordinary nilpotent cone of type $C$. In some respects, the former has properties which are better than those of the latter; the present work is another example of this. \vspace{5pt} 

The main result of this paper is a  description and a combinatorial enumeration of the irreducible components of exotic Springer fibres. In \cite{Spa}, Spaltenstein gives an explicit bijection between the irreducible components of Springer fibres in type $A$ and standard tableaux of the corresponding shape. Using his techniques, we will show that the irreducible components of exotic Springer fibres are in bijection with standard Young bitableaux, and explicitly describe an open, dense subset within each component. The fact that the cardinality of these two sets is equal follows from Kato's constructions in \cite{kat} and \cite{kato2} for the following reason: the top homology of the exotic Springer fibre has a basis given by the classes of the irreducible components, but it also carries an action of the Weyl group, and is isomorphic to the corresponding irreducible representation. Our proof of this fact has the advantage that it gives an explicit bijection, and is more elementary. \vspace{5pt} 

The irreducible components of Springer fibres in type $B$, $C$ and $D$ were first described by Spaltenstein in Section II.6 of \cite{irredspa}, but the combinatorics is quite subtle and involves signed domino tableaux (see also Section $3$ of van Leeuwen's thesis \cite{leeuwen} for an exposition, and also the simplified version given by Pietraho in \cite{pietraho}). In \cite{Ste}, Steinberg constructs a geometric Robinson-Schensted correspondence by looking at the irreducible components of the Steinberg variety in two different ways. In type A, this coincides with the classical RS correspondence defined combinatorially using the row bumping algorithm (see \cite{Ste}), but in types B, C and D, it is quite different, and was computed by van Leeuwen in \cite{leeuwen}. In the exotic setup, we obtain a variant of the Robinson-Schensted Correspondence in type C.  We give a combinatorial description in \cite{NRS17}, the sequel to this paper, and this is discussed briefly in Section \ref{sec:exoRS}. Note that this is different from the exotic Robinson-Schensted correspondence constructed by Henderson and Trapa in \cite{HT}. \vspace{5pt}

Let us now briefly summarize our paper; see Sections \ref{sec:notation} and \ref{sec:strategy} for more detail. The exotic nilpotent cone $\mathfrak{N}$ is the Hilbert nullcone of the $\text{Sp}_{2n}(\mathbb{C})$ representation $\mathbb{C}^{2n} \oplus \Lambda^{2}(\mathbb{C}^{2n})$. Like the ordinary case, the exotic nilpotent cone has a natural resolution $\pi: \widetilde{\mathfrak{N}} \rightarrow \mathfrak{N}$, with $\widetilde{\mathfrak{N}}$ being a vector bundle over the symplectic flag variety. Given $(v, x) \in \mathfrak{N}$, the exotic Springer fibre $\mathcal{C}_{(v,x)} = \pi^{-1}(v,x)$. We have: 
$$
\cC_{(v,x)}	= \{(0 \subset F_1 \subset \cdots \subset F_{2n-1} \subset \mathbb{C}^{2n}) \; | \;~F_i^\perp=F_{2n-i}, \text{dim}(F_i) =i, ~v\in F_n,~x(F_i)\subseteq F_{i-1}\}.
$$
Given $(v, x) \in \mathfrak{N}$ and a bipartition $(\mu, \nu) \in \mathcal{Q}_n$, define the {\it exotic type} $\text{eType}(v,x) = (\mu, \nu)$ if $(v,x)$ lies in the $\Sp_{2n}(\mathbb{C})$-orbit indexed by $(\mu,\nu)$. Let $T$ be a {\it standard Young bitableau} (see Example \ref{example:bitableau} for a complete definition). Define $\mathcal{C}_{(v,x)}^{T, \circ} \subset \mathcal{C}_{(v,x)}$ to be the subset of all flags $(F_{\bullet})$, such that for each $1 \leq i \leq n$, $\text{eType}(v+F_{n-i}, F_{n-i}^{\perp} / F_{n-i})$ is the bitableau obtained by deleting all entries of $T$ which are larger than $i$. Our main theorem states that the irreducible components of the exotic Springer fibre $\mathcal{C}_{(v,x)}$ are equidimensional, and are given by the closures $\overline{\mathcal{C}_{(v,x)}^{T, \circ}}$, as $T$ ranges over all standard bitableaux of shape $(\mu, \nu)$. It should be noted here that that fibre is not the union of the $\mathcal{C}_{(v,x)}^{T, \circ}$ and that taking closures is necessary. \vspace{5pt} 

The proof is inductive, and relies on analysing the projection map $p: \mathcal{C}_{(v,x)}^T \rightarrow \P(\ker(x))$, given by  $p(F_{\bullet}) = F_1$. The fibres of this projection map are easily seen to be either empty, or isomorphic to $\mathcal{C}_{(v+F_1, x|_{F_1^{\perp}/F_1})}^{T'}$, where $T'$ is obtained from $T$ by removing the box labelled $n$. The bulk of the proof lies with Proposition \ref{prop:B-irred}, where we describe the image of the map $p$; the main theorem essentially follows once we know the dimension of the image, and its irreducibility (since the total space of a fibre bundle with irreducible base and fibres is also irreducible). This strategy is roughly the same as that used by Spaltenstein for type A Springer fibres in \cite{Spa}; in that case, instead of the exotic type one simply looks at the Jordan type of the nilpotent $x$ on the subspace $F_i$, and understanding the image of the map $p$ is much easier. \vspace{5pt}

As an application, in Section \ref{sec:dimension2} we study the special case where the exotic Springer fibre has dimension two. Using our main theorem, we show that in this case the irreducible components are $\mathbb{P}^1$-bundles over $\mathbb{P}^1$, and classify them.  \vspace{5pt}

\subsection{Acknoledgements}
We are very much indebted to Anthony Henderson for many helpful suggestions, and N.S. would particularly like to thank him for supporting a research visit to Sydney where part of this work was carried out. We would also like to thank Syu Kato for a useful conversation on this subject. D.R. and V.N. are grateful to the Centre de Recherches Math\'ematiques and the Fields Institute for sponsoring the workshop on Infinite Dimensional Lie Theory: Algebra, Geometry and Combinatorics, where they discussed these ideas. N.S thanks A. Wilbert for useful discussions regarding the material in Section \ref{sec:dimension2}. N.S. is also grateful to the Heilbronn Institute for Mathematical Research, l'\'Ecole Polytechnique F\'er\'erale de Lausanne and Donna Testerman for supporting this research. V.N. is grateful to the University of Utah (in particular, Peter Trapa), and the University of Sydney (in particular, Ruibin Zhang and Gus Lehrer) for supporting this research. Finally, we thank the anonymous referee for pointing out a weakness in the proof of our main theorem in a previous draft.

\section{Background and Notation}\label{sec:notation}

We let $U$ be an $n$-dimensional vector space over $\C$ and we let $V=U \oplus U^{*}$. We endow $V$ with a symplectic (i.e. nondegenerate, skew-symmetric, bilinear) form thus: $$\langle (u,f),(u',f') \rangle:=f'(u)-f(u'),\text{ for }u,u' \in U\text{ and }f,f' \in U^{*}.$$
If $W\subseteq V$, we denote by $W^\perp$ its perpendicular space with respect to the form $\langle ~,~\rangle$. \vspace{5pt}

We define the symplectic group as the group of invertible linear transformations of $V$ preserving the form
$$
\Sp_{2n}=\Sp(V):=\{g\in \Gl(V)~|~\langle gv, gw \rangle = \langle v,w\rangle, ~\forall v,w\in V\}
$$
and we identify its Lie algebra as follows
$$
\spp_{2n}:=\Lie(\Sp_{2n})=\{x\in \End(V)~|~\langle xv,w\rangle + \langle v,xw\rangle =0, ~\forall v,w\in V\}.
$$
The adjoint action of $\Sp_{2n}$ on $\spp_{2n}$ is the restriction of the $\Sp_{2n}$-action on $\gl_{2n}=\End(V)$ given by conjugation. This action gives a direct sum decomposition of $\Sp_{2n}$-modules
$\gl_{2n}=\spp_{2n}\oplus \cS.$ We can also describe $\cS$ directly as
$$\cS=\{x\in \End(V)~|~\langle xv,w\rangle -\langle v,xw\rangle =0, ~\forall v,w\in V\}.$$
An easy fact following directly from this definition is that $\langle x^{i}v,x^{j}v \rangle =0$ for all $v \in V$ and all $i,j \geq 0$. 
Finally we let $\cN(\mathfrak{gl}_{2n}):=\{x\in\End(V)~|~x^k=0,\text{ for some }k\}$ be the \emph{nilpotent cone} of $GL_{2n}$.

\begin{definition}[(Exotic Nilpotent Cone)]\label{def:ex-nilcone}
The \emph{exotic nilpotent cone} is the (singular) variety $\mathfrak{N}:=V \times (\mathcal{S} \cap \mathcal{N}(\mathfrak{gl}_{2n}))$.
\end{definition}

We denote by $\cF(V)$ the variety of \emph{complete symplectic flags} in $V$, that is $F_\bullet\in\cF(V)$ is a sequence of subspaces
$$ F_\bullet=(0=F_0\subseteq F_1\subseteq \cdots F_n\subseteq \cdots \subseteq F_{2n-1}\subseteq F_{2n}=V)$$
such that $\dim(F_i)=i$ and $F_i^\perp=F_{2n-i}$.  \vspace{5pt}

There is a resolution of singularities
\begin{equation}\label{eq:resol-nilcone} \pi:\widetilde{\fN}\twoheadrightarrow \fN\end{equation}
where 
$$\widetilde{\mathfrak{N}}=\{(F_\bullet,(v,x))\in\cF(V)\times\fN~|~v\in F_n,~x(F_i)\subseteq F_{i-1}~\forall i=0,\ldots,2n\} $$
given by the projection
$$ \pi(F_\bullet,(v,x))=(v,x).$$
Notice that the map $\pi$ is equivariant for the natural diagonal action of $\Sp_{2n}$ on both $\widetilde{\fN}$ and $\fN$ with the explicit action on $\widetilde{\fN}$ given by $g\cdot (F_\bullet,v,x)=(gF_\bullet,gv,gxg^{-1})$.
\vspace{5pt}

A \emph{partition} of $n$ is a sequence $\lambda=(\lambda_1,\ldots,\lambda_k)$ of nonnegative integers with $\lambda_1\geq\ldots\geq\lambda_k$ and $\sum\limits_{i=1}^k\lambda_i=n$. We denote it by $\lambda \vdash n$ or $|\lambda|=n$. The \emph{length} of a partition $\lambda$, denoted by $\ell(\lambda)$ is the number of nonzero parts of $\lambda$. If we have two partitions $\mu$, $\nu$, we can define new partitions $\mu+\nu=(\mu_1+\nu_1,\ldots,\mu_s+\nu_s)$, and $\mu\cup\nu$ which is the unique partition obtained by reordering the sequence $(\mu_1,\ldots,\mu_r,\nu_1,\ldots,\nu_s)$ to make it nondecreasing. A \emph{bipartition} of $n$ is a pair $(\mu,\nu)$ of partitions such that $|\mu|+|\nu|=n$. We denote the set of all bipartitions of $n$ by $\cQ_n$. The set $\cQ_n$ is important for us because of the following result.

\begin{theorem}[{\cite[Thm 6.1]{AH}}] 
The orbits of $\Sp_{2n}$ on $\fN$ are in bijection with $\cQ_n$.
\end{theorem}
More precisely, following Section 2 and Section 6 of \cite{AH} we can say that, given a bipartition $(\mu,\nu)\in\cQ_n$, the corresponding orbit $\mO_{(\mu,\nu)}$ contains the point $(v,x)$ if and only if there is a {\it normal basis} of $V$ given by 
$$\{v_{ij}, v_{ij}^{*} \, | \, 1 \leq i \leq \ell(\mu+\nu), 1 \leq j \leq (\mu+\nu)_{i}=\lambda_i \},
$$ with $\langle v_{ij},v^*_{i'j'}\rangle=\delta_{i,i'}\delta_{j,j'}$, $v=\sum\limits_{i=1}^{\ell(\mu)}v_{i,\mu_i}$
and such that the action of $x$ on this basis is as follows: 
$$
xv_{ij}=\begin{cases} v_{i,j-1}  & \mbox{if} \quad j \geq 2 \\ 0  &\mbox{if} \quad j=1 \end{cases} \quad \quad \quad xv_{ij}^{*}=\begin{cases} v_{i,j+1}^{*}  & \mbox{if} \quad j \leq \mu_i+\nu_i-1 \\ 0  &\mbox{if} \quad  j=\mu_i+\nu_i. \end{cases}
$$ In particular the Jordan type of $x$ is $(\mu+\nu)\cup(\mu+\nu)$.

\begin{definition}\label{def:jordan}
If $x$ is a nilpotent transformation on a vector space $W$, we denote by $\Type(x,W)$ the Jordan type of $x$, which is a partition of $\dim(W)$.
If $(v,x)\in\mO_{(\mu,\nu)}$ as defined above, we say that the bipartition $(\mu,\nu)$ is the \emph{exotic Jordan type} of $(v,x)$ and we denote it by $\eType(v,x)=(\mu,\nu)$.
\end{definition}
Associated to a partition $\lambda$ there is a Young diagram consisting of $\lambda_i$ boxes on row $i$. We say that $\lambda$ is the \emph{shape} of the diagram. In the same way, a bipartition gives a pair of diagrams.

\begin{example} \label{example:boxes}
Consider the bipartition $(\mu,\nu)=((3,1),(2,2,1))$, and a point $(v,x)\in\mO_{(\mu,\nu)}$. Then we can represent the normal basis as

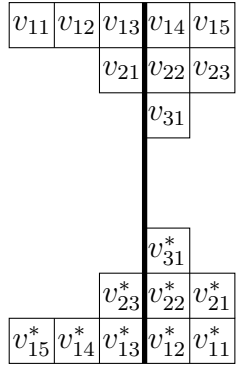
\begin{center}
\begin{tikzpicture}[scale=0.6]
\draw[-,line width=2pt] (0,4) to (0,-4);

\draw (-3,4) -- (2,4) -- (2,3) -- (-3,3) -- cycle;
\draw (-2,4) -- (-2,3);
\draw (-1,4) -- (-1,2) -- (2,2) -- (2,3);
\draw (1,4) -- (1,1) -- (0,1);

\node at (-2.5,3.5) {$v_{11}$};
\node at (-1.5,3.5) {$v_{12}$};
\node at (-0.5,3.5) {$v_{13}$};
\node at (0.5,3.5) {$v_{14}$};
\node at (1.5,3.5) {$v_{15}$};

\node at (-0.5,2.5) {$v_{21}$};
\node at (0.5,2.5) {$v_{22}$};
\node at (1.5,2.5) {$v_{23}$};

\node at (0.5,1.5) {$v_{31}$};


\draw (-3,-4) -- (2,-4) -- (2,-3) -- (-3,-3) -- cycle;
\draw (-2,-4) -- (-2,-3);
\draw (-1,-4) -- (-1,-2) -- (2,-2) -- (2,-3);
\draw (1,-4) -- (1,-1) -- (0,-1);

\node at (-2.5,-3.5) {$v_{15}^*$};
\node at (-1.5,-3.5) {$v_{14}^*$};
\node at (-0.5,-3.5) {$v_{13}^*$};
\node at (0.5,-3.5) {$v_{12}^*$};
\node at (1.5,-3.5) {$v_{11}^*$};

\node at (-0.5,-2.5) {$v_{23}^*$};
\node at (0.5,-2.5) {$v_{22}^*$};
\node at (1.5,-2.5) {$v_{21}^*$};

\node at (0.5,-1.5) {$v_{31}^*$};
\end{tikzpicture}
\end{center}
with the action of $x$ given by moving one block left (and zero if there is nothing further left), and $v=v_{13}+v_{21}$ is given by the sum of the boxes just left of the dividing wall on the upper half of the diagram. Notice that the wall divides the two diagrams corresponding to the partitions that form the bipartition (the one on the left of the wall is facing backwards) and that the bipartition is repeated twice.
\end{example}
Given a Young diagram, we can fill the boxes with positive integers to obtain a Young tableau. We call a Young tableau, of shape $\lambda\vdash n$, \emph{standard} if it contains the integers $1,\ldots, n$ and it is strictly increasing along rows and down columns. For a bipartition $(\mu^1,\mu^2)$, a \emph{bitableau} of shape $(\mu^1,\mu^2)$ is a pair $(T^1,T^2)$ where $T^i$ is a tableau of shape $\mu^i$, $i=1,2$. A bitableau of shape $(\mu,\nu)$ is standard if it contains the integers $1,2,\ldots,|\mu|+|\nu|$ and each tableau is increasing along rows and down columns. To match the conventions used in \cite{AH}, we will actually reverse left-to-right the first tableau of a bitableau so that both tableaux increase as they get farther from the centre. We denote the set of all {\it standard bitableaux }of shape $(\mu,\nu)$ by $\SYB(\mu,\nu)$.

\begin{example} \label{example:bitableau}
The following is a standard bitableau of shape $((3,1),(2,2,1))$.
$$ \left( \young(841,::3),\young(25,69,7)\right)$$
\end{example}
\begin{remark}
A standard bitableau of shape $(\mu,\nu)$ is the same thing as a \emph{nested} sequence of bipartitions ending at $(\mu,\nu)$ i.e. a sequence of bipartitions
$$(\varnothing, \varnothing), (\mu^{(1)}, \nu^{(1)}), \ldots, (\mu^{(s)},\nu^{(s)})=(\mu,\nu)$$
such that $(\mu^{(i+1)},\nu^{(i+1)})$ is obtained from $(\mu^{(i)},\nu^{(i)})$ by adding one box (increasing one of the parts by one) either on the left or on the right tableau. The identification is given by tracing the order in which the boxes are added according to the increasing sequence of numbers $1,2,\ldots,|\mu|+|\nu|$.
\end{remark}
\begin{example}
The standard bitableau $\left(\young(32,:5),\young(1,4)\right)$ corresponds to the nested sequence
$$ \left(\varnothing,\varnothing\right),\left(\varnothing,\yng(1)\right),\left(\yng(1),\yng(1)\right),\left(\yng(2),\yng(1)\right),\left(\yng(2),\yng(1,1)\right),\left(\young(\hfil\hfil,:\hfil),\yng(1,1)\right).$$
\end{example}
\begin{definition}
Given $(v,x)\in\fN$, we consider the \emph{exotic Springer fibre}
$$\cC_{(v,x)}:=\pi^{-1}(v,x)=\{F_\bullet=(0\subseteq F_1\subseteq\cdots\subseteq F_{2n}=V~|~F_i^\perp=F_{2n-i},~v\in F_n,~x(F_i)\subseteq F_{i-1}\}.$$
It is clear that if $\eType(v,x)=\eType(v',x')$, then there is an isomorphism of varieties $\cC_{(v,x)}\simeq\cC_{(v',x')}$ given by the $\Sp_{2n}$-action. 
\end{definition}
Let $(v,x)\in\fN$, if $F_\bullet\in\cC_{(v,x)}$, for all $i=0,\ldots,n$, since $F_i$ and $F_i^\perp=F_{2n-i}$ are invariant under $x$, we can consider the restriction of $x$ to $F_i^\perp / F_i$, which is a vector space of dimension $2(n-i)$. Furthermore, it can be easily verified that $x\in\cN( F_i^\perp / F_i)\cap\cS (F_i^\perp/F_i)$.

\begin{definition}\label{def:Phi}
Let $(v,x)\in\fN$, $\eType(v,x)=(\mu,\nu)$, and let $T\in\SYB(\mu,\nu)$.
Define the following maps, and the subsets $\mathcal{C}_{(v,x)}^{T} \subset \mathcal{C}_{(v,x)}$: 
\begin{align*} 
\Phi^i:\cC_{(v,x)}\to \cQ_{n-i} &\qquad F_\bullet \mapsto \eType\left(v+F_i, x|_{F_i^\perp/F_i}\right); \\ 
\Phi:\cC_{(v,x)}\to \prod_{i=0}^n\cQ_i &\qquad F_\bullet\mapsto \left(\Phi^n(F_\bullet),\ldots,\Phi^0(F_\bullet)\right); \\
\mathcal{C}_{(v,x)}^{T} &:= \overline{\Phi^{-1}(T)}. 
\end{align*} 
\end{definition} 
\begin{definition}\label{def:bdim}
Let $(\mu,\nu)$ be a bipartition (with $\mu + \nu = \lambda)$, define 
$$b(\mu,\nu):=2N(\lambda)+|\nu|, \text{ where } N(\lambda)=\sum_{i\geq1}(i-1)\lambda_i $$ 
\end{definition}
\begin{remark}
The quantity $N(\lambda)$ is the dimension of the irreducible components of the Springer fibre in type A corresponding to the partition $\lambda$.
\end{remark} Our main theorem is the following: 
\begin{theorem}\label{mainthm}
Let $(v,x)\in\fN$, with $\eType(v,x)=(\mu,\nu)$, where $(\mu,\nu)\in\cQ_n$. Then the irreducible components of $\cC_{(v,x)}=\pi^{-1}(v,x)$ are precisely:
$$\{ \cC_{(v,x)}^T ~|~T\in\SYB(\mu,\nu)\}.$$ They all have the same dimension $b(\mu, \nu)$. \end{theorem}
In the case of the Springer fibres for groups of type A, the appropriate analogue of the map $\Phi$ from Definition \ref{def:Phi} gives an increasing sequence of partitions, which is equivalent to a standard tableau. Unfortunately in our case, it is not true in general that the image of $\Phi$ consists of nested sequences of bipartitions (i.e. standard bitableaux), as the following example shows.
\begin{example}\label{ex:notnested}
Take $n=2$, and consider the bipartition $(\mu,\nu)=((1,1),\varnothing)=\left(\yng(1,1),\varnothing\right)$. According to Definition \ref{def:jordan}, if $(v,x)\in\cC_{(\mu,\nu)}$, then $x=0$ and $v\neq 0$. Now, we let $F_1=\mathbb{C}v$, and we let $F_2$ be any $2$-dimensional subspace of $V$ such that $F_1\subseteq F_2\subseteq F_1^{\perp}$ and set $F_{\bullet}=(0 \subset F_1 \subset F_2 \subset F_{1}^{\perp} \subset \C^4)$. We then have $(F_\bullet,v,x)\in \mathcal{C}_{v,x}$. To compute the image under the map $\Phi$ the cases of $\Phi^0$ and $\Phi^2$ are trivial, so we need to consider what is the bipartition of $1$ given by $\Phi^1(F_\bullet)$. It is easy to see that $\eType\left(v+F_1,x|_{F_1^\perp/F_1}\right)=(\varnothing,1)$ because $v\in F_1$ so $v+F_1=0$ in $F_1^{\perp}/F_1$. Hence we get the sequence
$$(\varnothing,\varnothing), (\varnothing,(1)),((1,1),\varnothing) \qquad
\text{or equivalently} \qquad 
(\varnothing,\varnothing), \left(\varnothing,\yng(1)\right),\left(\yng(1,1),\varnothing\right)$$
which is not nested.
\end{example}

\section{Strategy of the Proof}\label{sec:strategy}
To prove our main theorem, which describes the irreducible components of the varieties $\mathcal{C}_{(v,x)}$, we will use some properties of the map $\Phi$ from Definition \ref{def:Phi}. 
We already know from Example \ref{ex:notnested} that the image of $\Phi$ does not consist excusively of standard bitableaux (i.e. nested sequences of bipartitions), but it does give standard bitableaux in most cases. In fact those are the only cases we need, as the following proposition tells us. The proof of the main theorem, which will be given in Section \ref{section:steinbergvariety}, will rely on the following fact. 
\begin{proposition} \label{prop:nested}
Let $(\mu,\nu)\in\cQ_n$, let $T$ be a standard bitableau of shape $(\mu,\nu)$ and let $(v,x)\in\mO_{(\mu,\nu)}$. Then $\Phi^{-1}(T)$ is an irreducible subvariety of $\cC_{(v,x)}$ of dimension $b(\mu,\nu)$. \end{proposition}

\begin{proof}
We will prove this proposition by induction on $n$. For $n=0$, the statement is trivial (all the varieties involved are one single point), so now assume that $n\geq 1$.
Fix $(\mu,\nu)$, $T$ and $(v,x)$ as in the statement. We consider the projection
\begin{eqnarray*}
p:\Phi^{-1}(T) 	&\longrightarrow& \P(\ker(x) \cap (\C[x]v)^{\perp}); \\
	F_\bullet 	&\mapsto& F_1.
\end{eqnarray*}
Let $T'$ be the standard bitableau obtained from $T$ by removing the box with the number $n$, and let $(\mu',\nu')\in\cQ_{n-1}$ be the shape of $T'$. Notice that $(\mu',\nu')=\Phi^1(F_\bullet)=\eType\left(v+F_1,x|_{F_1^\perp/F_1}\right)$. Observe that  
$$ \cB_{(\mu',\nu')}^{(\mu,\nu)}:=\im p=\{F_1 \subset \ker(x) \cap (\mathbb{C}[x]v)^{\perp}\, | \, \eType(v+F_1,x_{|F_{1}^{\perp}/F_1})=(\mu',\nu')\, \}. $$

For any $F_1\in\im p$, we then have that 
\begin{align}
\notag p^{-1}(F_1)&=\{F'_\bullet\in \Phi^{-1}(T)~|~F'_1=F_1\} \\
\label{eq:p-proj} &\stackrel{\alpha}{\simeq} \left\{\overline{F}_\bullet\in\cC_{\left(v+F_1,x_{|F_1^\perp/F_1}\right)}\subseteq\cF(F_1^\perp/F_1)~|~(\Phi^{n-1}(\overline{F}_\bullet),\ldots,\Phi^0(\overline{F}_\bullet))=T'\right\} 
\end{align}
where the isomorphism $\alpha$ is given by 
$$\alpha(F'_0,F'_1,\ldots,F'_{2n-1},F'_{2n})=(\overline{F}_0,\overline{F}_1,\ldots,\overline{F}_{2n-3},\overline{F}_{2n-2});\quad \overline{F}_i=F'_{i+1}/F_1\quad\forall i=0,\ldots,2n-2.$$
In conclusion, the map $p$ is a fibre bundle, and from \eqref{eq:p-proj}, it follows that the fibres are $p^{-1}(F_1)\simeq \Phi^{-1}(T')$, which by inductive hypothesis is an irreducible variety of dimension $b(\mu',\nu')$.
Thus the proof will be a consequence of the following result.
\end{proof}
\begin{proposition}\label{prop:B-irred} For all $(\mu,\nu)\in\cQ_n$ and for all $(\mu',\nu')\in\cQ_{n-1}$ such that $(\mu',\nu')$ is obtained from $(\mu,\nu)$ by removing one box,
the variety $ \mathcal{B}_{(\mu',\nu')}^{(\mu,\nu)}$ is irreducible of dimension $b(\mu,\nu)-b(\mu',\nu')$.
\end{proposition}
The proof of Proposition \ref{prop:B-irred} will be given in Section \ref{sec:proof} of the paper.
The following result, which combines \cite[Thm 1 and Cor 1]{T} and \cite[Thm 6.1]{AH} will be very useful for us in this regard.
\begin{theorem}[Travkin, Achar-Henderson]\label{thm:T-AH} Let $(v,x)\in\fN(W)$, then
$\eType(v,x)=(\mu,\nu)$ if and only if 
\begin{eqnarray*}
\Type(x,W) 		&=& (\mu_1+\nu_1, \mu_1+\nu_1,\mu_2+\nu_2,\mu_2+\nu_2,\ldots) \quad \text{and} \\
\Type(x,W/(\C[x]v))	&=& (\mu_1+\nu_1, \mu_2+\nu_1,\mu_2+\nu_2,\mu_3+\nu_2,\ldots).
\end{eqnarray*}
\end{theorem}
We want to apply this theorem to the space $W=F_1^\perp/F_1$, so we will need to calculate the Jordan types of the induced nilpotent on the spaces 
$$F_{1}^{\perp}/F_{1} ~ \, \text{ and } \, ~ (F_{1}^{\perp}/F_1)/((\C[x]v+F_1)/F_1)=F_1^\perp/(\C[x]v+F_1).$$


\section{Calculating Jordan Types} \label{sec:jordan}
Recall that $V$ is a $2n$-dimensional symplectic vector space over $\C$, as defined in Section \ref{sec:notation}. Throughout this section we fix $(v,x)\in\fN$, such that  $\eType(v,x)=(\mu,\nu)$ where $(\mu,\nu)\in\cQ_n$ is a bipartition of $n$. In particular this means that $\Type(x,V)=\lambda\cup\lambda$ where $\lambda:=\mu+\nu$. 
\vspace{5pt}

Whenever we have an $x$-invariant subspace $W\subset V$, by abuse of notation we will often denote the induced linear transformations $x|_{W}$ and $x|_{V/W}$ also by $x$. It should be clear at all times which vector space we are working with.
\vspace{5pt}

We start by giving two general lemmas regarding the Jordan types of induced nilpotent endomorphisms on hyperplanes and on quotients by lines. The proofs of these lemmas can be found in \cite{Spa} or \cite[Prop. 1.4]{leeuwen2}. 
\vspace{5pt}

\begin{lemma} \label{lemma:dim1} Given a nilpotent endomorphism $x$ of $V$ with Jordan type $\lambda$, and a $1$-dimensional subspace $L$ with $L \subseteq \ker(x)$, suppose that $j$ is maximal such that $L \subseteq \text{ker}(x) \cap \text{im}(x^{j-1})$. Then the Jordan type of $x$ on $V/L$ is the partition obtained by removing the corner box at the bottom of the $j$-th column.
\end{lemma}

\begin{lemma} \label{lemma:codim1}  Given a nilpotent endomorphism $x$ of $V$ with Jordan type $\lambda$, and a codimension $1$ subspace $W$ with $W \supseteq \text{im}(x)$, suppose that $j$ is maximal such that $W \supseteq \text{im}(x) + \ker(x^{j-1})$. Then the Jordan type of $x$ on $W$ is the partition obtained by removing the corner box at the bottom of the $j$-th column. 
\end{lemma}

We fix a Jordan `normal' basis for $x$ on $V$ as in Section \ref{sec:notation} 
$$\{v_{ij}, v_{ij}^{*} \, | \, 1 \leq i \leq \ell(\mu+\nu), 1 \leq j \leq (\mu+\nu)_{i}=\lambda_i \},$$  and we write $$v=\sum_{i=1}^{\ell(\mu)}v_{i,\mu_i},$$ which, in terms of the diagram in Example \ref{example:boxes} is the vertical sum of the boxes immediately to the left of the dividing wall. Let $F_1\in\P(\ker (x)\cap(\C[x]v)^\perp$ be a one-dimensional space, we write $F_1=\C v_1$, where 
\begin{equation} \label{equation:v1} 
0\neq v_1=\alpha_1v_{1,1}+\ldots + \alpha_{\ell(\lambda)}v_{\ell(\lambda),1}+\beta_{1}v_{1,\lambda_1}^{*}+\ldots+\beta_{\ell(\lambda)}v_{\ell(\lambda),\lambda_{\ell(\lambda)}}^{*}.
\end{equation}

\begin{definition}\label{def:m-sets}
In \eqref{equation:v1}, let $m$ be the largest integer such that $\alpha_i \neq 0$ or $\beta_i \neq 0$ and we define three sets of indices: 
\begin{eqnarray*}
\Lambda_{m}	&=& \{1 \leq i \leq \ell(\lambda) \, | \, \lambda_i=\lambda_m\}, \\
\Gamma_{m}	&=& 	\{ 1 \leq i \leq \ell(\lambda) \, | \,  \mu_i=\mu_m\}, \\
\Delta_{m}		&=& \{1 \leq i \leq \ell(\lambda)\, | \,  \nu_i=\nu_m\}.
\end{eqnarray*}
\end{definition}
We adopt the convention that if $m>\ell(\mu)$ (respectively $m > \ell(\nu)$), then $\Gamma_m=\varnothing$ (respectively $\Delta_m=\varnothing$).
\begin{example}
Let $\mu=(3^3,2)$, $\nu=(3,2^2,1^2)$, $\lambda=(6,5^2,3,1)$. Here $\Gamma_1=\Gamma_2=\Gamma_3=\{1,2,3\}$, $\Gamma_4=\{4\}$, $\Gamma_5=\{5\}$, while $\Delta_1=\{1\}$, $\Delta_2=\Delta_3=\{2,3\}$, $\Delta_4=\Delta_5=\{4,5\}$.
\end{example}

\begin{definition} \label{def:m}
For $m$ maximal such that either $\alpha_i \neq 0$ or $\beta_i \neq 0$ as in Definition \ref{def:m-sets}, we define $\ovm=\max\Lambda_m$.
\end{definition}

\begin{remark}
Let $\lambda^{tr}$ be the transpose partition of $\lambda$, defined by the property that $\lambda^{tr}_i=|\{j~|~\lambda_j\geq i\}|$. It is worth noting here that $\lambda_{\lambda_m}^{tr}=\overline{m}$.
\end{remark}
\begin{remark}\label{rem:perp}
If $v_1$ is as in \eqref{equation:v1}, notice that, for all $r\geq 1$,
$$ \langle v_1, x^rv\rangle =  \langle x^r v_1 ,v\rangle = \langle 0,v\rangle=0$$
and
$$ 
\langle v_1,v\rangle  =\left\langle \sum_{i=1}^{\ell(\lambda)}\alpha_i v_{i,1}+\sum_{i=1}^{\ell(\lambda)}\beta_i v_{i,\lambda_i}^*,\sum_{i=1}^{\ell(\mu)}v_{i,\mu_i}\right\rangle =- \sum_{i:\,\lambda_i=\mu_i}\beta_i.
$$
It follows that $v_1\in(\C[x]v)^\perp$ if and only if $\displaystyle \sum_{\substack{i :~ \nu_i=0, \\ ~~\mu_i>0}} \beta_i= 0$.
\end{remark}

\subsection{Jordan Type of $x$ on $V_{1}^{\perp}/V_1$: maximal $m$} \label{section:m}

\begin{proposition} \label{prop:duplicate}
Let $F_1=\C{v_1}$ with $v_1$ as in \eqref{equation:v1}, and $m,\ovm$ as in Definition \ref{def:m}. Then the Jordan type of $x$ on $F_{1}^{\perp}/F_1$ is the duplicated partition $\lambda' \cup \lambda'$ where $\lambda'$ is obtained from $\lambda$ by decreasing $\lambda_{\overline{m}}$ by $1$. 
That is, comparing $\lambda$ with $\lambda'$ we have 
\begin{displaymath}
\lambda'_i= 
\begin{cases} 
\lambda_{i} 						& for \quad  i \neq \overline{m}, \\ 
\lambda_{\overline{m}}-1 	& for \quad  i =\overline{m}.
\end{cases}
\end{displaymath}
\end{proposition}

\begin{proof}
By Lemma \ref{lemma:dim1} we may calculate $\Type(x,V/F_1)$. This is given by the maximal $k$ such that $F_{1}$ is contained in $x^{k-1}(V)$, whence the last box in the $k$-th column of the duplicated partition $\lambda \cup \lambda$ is removed. Here, by definition of $m$, is it easily seen that $F_{1} \subseteq x^{\lambda_m-1}(V)$ but $F_1 \not\subseteq x^{\lambda_m}(V)$. Therefore $\Type(x,V/F_{1})= \lambda \cup \lambda'$.  \vspace{5pt}

We now calculate $\Type(x,F_{1}^{\perp}/F_{1})$. By Lemma \ref{lemma:codim1}, this is given by the maximal $l$ such that $F_{1}^{\perp}/F_{1} \supset \ker(x_{|V/F_{1}}^{l-1})$, whence the last box in the $l$-th row of the partition $\lambda \cup \lambda'$ is removed. Translating this back to a condition on $F_{1}$, we deduce that we require the maximal $l$ such that $F_{1}^{\perp} \supseteq (x^{l-1})^{-1}(F_{1})$. It is clear that any vector that maps onto $v_1$ by $x$ is contained in $F_{1}^{\perp}$, so we just need to consider the kernels of powers of $x$. Again by definition of $m$, we see that $F_{1}^{\perp} \supset \ker(x^{\lambda_{m}-1})$ but $F_{1}^{\perp} \not\supset \ker(x^{\lambda_{m}})$. Therefore the Jordan type of $x$ on $F_{1}^{\perp}/F_{1}$ is $\lambda' \cup \lambda'$.
\end{proof}

\subsection{Jordan type of $x$ on $V/\mathbb{C}[x]v$}
We have $V=U\oplus U^{*}$, where $\C[x]v \subseteq U$ and so $V/\mathbb{C}[x]v \simeq U/\mathbb{C}[x]v\oplus U^{*}$.  Lemma 2.5 from \cite{AH} gives us a Jordan basis for $x$ on $U/\mathbb{C}[x]v$. 

\begin{lemma}\cite[Lemma 2.5]{AH}\label{lemma:AH}
\label{trav} The Jordan type of $x$ on $U/\mathbb{C}[x]v$ is $(\mu_2+\nu_1,\mu_3+\nu_2, \ldots)$ and so the Jordan type of $x$ on $V/\mathbb{C}[x]v$ is $\rho=(\mu_1+\nu_1,\mu_2+\nu_1,\mu_2+\nu_2,\mu_3+\nu_2, \ldots)$.
\end{lemma}

Remember that to use Theorem \ref{thm:T-AH}, we want to find the Jordan type of $x$ on $F_{1}^{\perp}/(\C[x]v+F_{1})$. To answer this question we will first compute $\Type(x, V/(\C[x]v+F_1))$ starting from $\Type(x,V/\C[x]v)$, from that we will then be able to compute $\Type(x,F_1^\perp/(\C[x]v+F_1))$. 

\subsection{Jordan type on $V/(\mathbb{C}[x]v+F_1)$: maximal $k$}
As in Lemma \ref{lemma:AH}, we let $\rho=(\rho_1,\rho_2,\ldots)=\Type(x,V/\C[x]v)$. We then have, for all $i\geq 1$ 
$$
\rho_{2i}=\mu_{i+1}+\nu_{i}\quad\text{ and }\quad\rho_{2i-1}=\mu_{i}+\nu_{i}.
$$ 

\begin{remark}\label{rem:F1-in-Cxv}
If $F_1$ is contained in $\C[x]v$, then $V/(\mathbb{C}[x]v+F_1) \simeq V/\mathbb{C}[x]v$ and $\Type(x,V/(\mathbb{C}[x]v+F_1))=\Type(x,V/\mathbb{C}[x]v)=\rho$. So we will now assume for the rest of this subsection that  $F_1 \not\subseteq \C[x]v$.
\end{remark}

If $F_1 \not\subseteq \C[x]v$, since $V/(\mathbb{C}[x]v+F_1)$ is a quotient of $V/\mathbb{C}[x]v$ by the 1-dimensional subspace $(\mathbb{C}[x]v+F_1)/\C[x]v$, we can use Lemma \ref{lemma:dim1}. 

\begin{lemma}\label{lemma:maxk1}
If $F_1 \not\subseteq \C[x]v$, the Jordan type of $x$ on $V/(\mathbb{C}[x]v+F_1)$ is determined by the maximum $k$ such that $F_1 \subset \ker(x) \cap (\im(x^{k-1})+\C[x]v)$. It is obtained from $\rho=\Type(x,V/\mathbb{C}[x]v)$, by deleting the last box in the bottom of the $k$-th column. 
\end{lemma}

\begin{proof}
By Lemma \ref{lemma:dim1}, we require the maximal $k$ such that $$(\mathbb{C}[x]v+F_1)/\C[x]v \subset \ker(x_{|V/\C[x]v}) \cap \im(x_{|V/\C[x]v}^{k-1}).$$   
Since $ \ker(x_{|V/\C[x]v}) =x^{-1}(\mathbb{C}[x]v)$ and $\im(x_{|V/\C[x]v}^{i})=\im(x^i)+\mathbb{C}[x]v$, this above condition can be written as the maximum $k$ such that: 
$$
\mathbb{C}[x]v+F_1 \subset x^{-1}(\mathbb{C}[x]v) \cap( \im(x^{k-1}) + \C[x]v). 
$$
Since $\C[x]v$ is clearly contained in the right hand side, and that $F_1 \subseteq \ker(x) \subseteq x^{-1}(\mathbb{C}[x]v)$,  we can reduce this condition to $F_1 \subset \ker(x) \cap (\im(x^{k-1})+\C[x]v).$ 
\end{proof}

We can actually be explicit regarding what the maximal $k$ in Lemma \ref{lemma:maxk1} is. Again, let $m$ be the maximal $i$ such that $\alpha_i \neq 0$ or $\beta_i \neq 0$ as in Definition \ref{def:m}. Since $F_1 \subseteq \ker(x) \cap \im(x^{\lambda_{m-1}})$, we clearly have that $k\geq \lambda_m$. 

\begin{definition}\label{def:m-al-be}
For $m$ as above, define the following (possibly empty) sets:
\begin{eqnarray*}
\Gamma_m^\alpha	&:=&	\{i\in\Gamma_m~|~\alpha_j=\alpha_m\neq 0,\quad\forall~i\leq j\leq\max\Gamma_m\}, \\
\Gamma_m^\beta	&:=&	\{i\in\Gamma_m~|~\beta_j=0,\quad\forall~i\leq j\leq\max\Gamma_m\}.
\end{eqnarray*}
We define also $m^\alpha=\min\Gamma_m^\alpha$, $m^\beta=\min\Gamma_m^\beta$ if those sets are nonempty.
\end{definition}
\begin{proposition} \label{prop:maxk}
Suppose $k$ is maximal such that $F_1 \subset \ker(x) \cap (\im(x^{k-1})+\C[x]v)$, $m$ as above and $m^\alpha$, $m^\beta$ as in Definition \ref{def:m-al-be}. Then 
$$k=
\begin{cases} 
\lambda_m & \text{ if }\qquad\Gamma_m^\alpha=\varnothing\quad \text{ or }\quad\Gamma_m^\beta=\varnothing; \\
\mu_m+\nu_{m'-1} 	& \text{ if }\qquad \Gamma_m^\alpha\neq\varnothing\neq\Gamma_m^\beta\quad\text{ where }\quad m'=\max\{m^\alpha,m^\beta\}.
\end{cases}
$$
\end{proposition}
\begin{proof}
We have already remarked that $F_1 \subset \im(x^{\lambda_m-1})$, hence $k\geq \lambda_m$. Suppose $k > \lambda_m$. Then the vectors $v_{i,1}$ and $v_{i,\lambda_i}^{*}$ for $i \in \Lambda_m$ lie outside $\im(x^{k-1})$ and so therefore must be obtained by adding the subspace $\C[x]v$. However, it is easily seen that the vectors $v_{i,\lambda_i}^{*}$ for $i \in \Lambda_m$ cannot be obtained by adding $\C[x]v$ and so $\beta_i=0$ for all $i \in \Lambda_m$ which implies that $\Gamma_m^\beta\neq\varnothing$. Moreover, the only vector contained in $\C[x]v$ that is relevant to our discussion is $\sum_{i \leq \max\Gamma_m} v_{i,\mu_i-\mu_m+1}$. Notice that, since $v_{\max\Gamma_m,1}\not\in\im(x^{k-1})$, if $\alpha_{\max\Gamma_m}=0$, then $v_1\not\in\im(x^{k-1})+\C\left(\sum_{i \leq \max\Gamma_m} v_{i,\mu_i-\mu_m+1}\right)$ for any $k$, which is impossible. Therefore $\alpha_{\max\Gamma_m}\neq 0$, which implies that $m=\max\Gamma_m$ and $\Gamma_m^\alpha\neq\varnothing$. 
Write 
$$ v_1=v'+v''+v'''$$
where
$$v'=\sum_{i\in\Gamma_m^\alpha}\alpha_iv_{i,1};\qquad v''=\sum_{i<m^\alpha}\alpha_i v_{i,1};\qquad v'''=\sum_{i<m^\beta}\beta_i v^*_{i,\lambda_i}.$$
By the way they are defined, we have $v_1\in\im(x^{k-1})+\C[x]v$ if and only if $v',v'',v'''\in\im(x^{k-1})+\C[x]v$.
Notice that by definition of $\Gamma_m^\alpha$, since all the coefficients for the corresponding indices are the same, we can write 
\begin{align*}v'&=\sum_{i\in\Gamma_m^\alpha}\alpha_iv_{i,1}=\sum_{i\in\Gamma_m^\alpha}\alpha v_{i,1}\\
&=\alpha\left(\sum_{i \leq \max\Gamma_m} v_{i,\mu_i-\mu_m+1}\right)-\alpha\left(\sum_{i<m^\alpha}v_{i,\mu_i-\mu_m+1}\right).
\end{align*}
 We then have $v'\in\im(x^{k-1})+\C[x]v$ if and only if $v_{i,\mu_i-\mu_m+1}\in\im(x^{k-1})$ for all $i<m^\alpha$, i.e. 
\begin{align*} 
k-1&\leq \lambda_i-(\mu_i-\mu_m+1) \\
k &\leq \mu_i+\nu_i-\mu_i+\mu_m \\
k &\leq \mu_m+\nu_i.
\end{align*}
Similarly, $v''\in\im(x^{k-1})+\C[x]v$ if and only if $v''\in\im(x^{k-1})$ if and only if $k\leq \lambda_i$ for all $i<m^\alpha$. Clearly this condition is redundant, since for all $i<m^\alpha$ we have $\mu_m+\nu_i\leq \mu_i+\nu_i=\lambda_i$. Finally, $v'''\in\im(x^{k-1})+\C[x]v$ if and only if $v'''\in\im(x^{k-1})$ if and only if $k\leq \lambda_i$ for all $i<m^\beta$. Putting the conditions together, we get indeed that $k\leq\mu_m+\nu_i$ where $i$ is maximal such that $i<m^\alpha$ or $i<m^\beta$, which proves the proposition.
\end{proof}

\begin{remark}
Notice that in the second case of Proposition \ref{prop:maxk}, we always have $m'> 1$. This is because if $m'=1$, then $m^\alpha=m^\beta=1$, which means that $v_1=\sum_{i\in\Gamma_m}\alpha v_{i,1}\in\C[x]v$ but we have assumed $F_1\not\subset\C[x]v$. Also note that it is possible that $\mu_m+\nu_{m'-1}$ can be equal to $\lambda_m$ for certain partitions $\mu$ and $\nu$. 
\end{remark}

\begin{example}
Consider the bipartition $(\mu, \nu)=((4,2^5,1),(3,2,1^2))$, so that $\lambda=(7,4,3^2,2^2,1)$. We may picture this below:
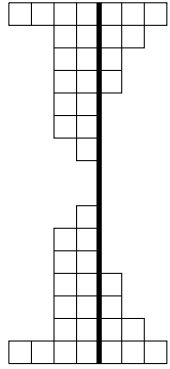
\begin{center}
\begin{tikzpicture}[scale=0.3]
\draw[-,line width=2pt] (0,8) to (0,-8);
\draw (-4,8) -- (-4,7) -- (3,7) -- (3,8) -- (-4,8);
\draw (-3,8) -- (-3,7);
\draw (-2,6) -- (2,6) -- (2,8);
\draw (-2,5) -- (1,5);
\draw (-2,4) -- (1,4) -- (1,8);
\draw (-2,8) -- (-2,2) -- (0,2);
\draw (-2,3) -- (0,3);
\draw (-1,8) -- (-1,1) -- (0,1);


\draw (-4,-8) -- (-4,-7) -- (3,-7) -- (3,-8) -- (-4,-8);
\draw (-3,-8) -- (-3,-7);
\draw (-2,-6) -- (2,-6) -- (2,-8);
\draw (-2,-5) -- (1,-5);
\draw (-2,-4) -- (1,-4) -- (1,-8);
\draw (-2,-8) -- (-2,-2) -- (0,-2);
\draw (-2,-3) -- (0,-3);
\draw (-1,-8) -- (-1,-1) -- (0,-1);

\end{tikzpicture}
\end{center}

In this case, $\C[x]v$ is $4$-dimensional, being the span of the vertical sum of boxes to the left of the wall on the top half of the diagram. Suppose 
$$\footnotesize{
v_1=\alpha_{1}v_{1,1}+\alpha_{2}v_{2,1} + \alpha_{3}v_{3,1}+\alpha_{4}v_{4,1} + \alpha_{5}v_{5,1} + \alpha_{6}v_{6,1}+ \beta_{6}v_{6,2}^{*}+ \beta_{5}v_{5,2}^{*}+\beta_{4}v_{4,3}^{*}+\beta_{3}v_{3,3}^{*}  + \beta_{2}v_{4,4}^{*}  + \beta_{1}v_{1,7}^{*}},
$$ with $\alpha_6 \neq 0$ or $\beta_6 \neq 0$.  Observe that by Remark \ref{rem:perp}, since $F_1\subseteq(\C[x]v)^{\perp}$, we require $\beta_5+\beta_6=0$.  In this case we have $m=6$ with $\lambda_m=\mu_m=2$; $\Gamma_m=\{2,3,4,5,6\}$, $\Lambda_m=\{5,6\}$ and $\Delta_m=\{5,6,7\}$. 
Therefore we have
$$
F_1 \subset  
\begin{cases} 
\im(x^4)+ \C[x]v  	& \text{if} \quad  \alpha_2=\ldots =\alpha_6 \neq 0 \quad \text{and} \quad \beta_2=\ldots=\beta_6=0, \\ 
\im(x^3)+ \C[x]v		& \text{if} \quad  \alpha_3=\alpha_4=\alpha_5 =\alpha_6 \neq 0 \quad \text{and} \quad \beta_3=\beta_4=\beta_5=\beta_6=0, \\ 
\im(x^2)+ \C[x]v  	& \text{if} \quad \alpha_5=\alpha_6 \neq 0 \quad \text{and} \quad \beta_5=\beta_6=0, \\ 
\im(x)+ \C[x]v 		& \text{in any case}.					
\end{cases}
$$
Therefore the maximal $k$ such that $V_1 \subset \im(x^{k-1}) + \C[x]v$ is: 
$$
k= \begin{cases} 
5 (=\mu_6+\nu_1)	& \text{if} \quad  \alpha_2=\ldots =\alpha_6 \neq 0 \quad \text{and} \quad \beta_2=\ldots=\beta_6=0, \text{or} \\  
4(=\mu_6+\nu_2)	& \text{if} \quad  \alpha_3=\alpha_4=\alpha_5 =\alpha_6 \neq 0 \quad \text{and} \quad \beta_3=\beta_4=\beta_5=\beta_6=0, \text{or} \\ 
3(=\mu_6+\nu_4) 	& \text{if} \quad \alpha_5=\alpha_6 \neq 0 \quad \text{and} \quad \beta_5=\beta_6=0, \text{or}\\ 
2(=\lambda_6) 		& \text{otherwise}.	
\end{cases}
$$
 
\end{example}

\subsection{The Jordan Type of $x$ on $F_{1}^{\perp}/(\C[x]v+F_{1})$: maximal $l$}

\begin{lemma} \label{lemma:preimagesum}
Let $l$ be a non-negative integer with $l < \lambda_m$. Then $$(x^{l})^{-1}(\C[x]v+F_1)=(x^{l})^{-1}(\C[x]v) + (x^{l})^{-1}(F_1).$$ 
\end{lemma}
\begin{proof}
It is clear that $(x^{l})^{-1}(\C[x]v+F_1) \supseteq (x^{l})^{-1}(\C[x]v) + (x^{l})^{-1}(F_1),$ so we prove the reverse inclusion. Suppose $w \in V$ such that $x^{l}(w) \in \C[x]v+F_1$. Since an obvious basis for $ \C[x]v+F_1$ is $\{v,x(v), \ldots x^{\mu_1-1}(v),v_1\}$, we may then write $$x^{l}(w)=\sum_{i=0}^{\mu_1-1}a_ix^{i}v + bv_1,$$ for $a_i, b \in \C.$ Since $l \leq \lambda_{m} -1$ we have $F_1  \subseteq \im(x^{\lambda_m-1}) \subseteq \im(x^{l})$ and so there exists a $y \in V$ such that $x^{l}(y)=bv_1$. We then find that $x^{l}(w-y)=\displaystyle\sum_{i=0}^{\mu_1-1}a_ix^{i}v$ and so $w-y \in (x^{l})^{-1}(\C[x]v)$, hence $w=y -(y-w) \in (x^{l})^{-1}(\C[x]v) + (x^{l})^{-1}(F_1),$ and we are done. 
\end{proof}
The space $F_{1}^{\perp}/(\C[x]v+F_{1})$ is a codimension 1 subspace of $F/(\C[x]v+F_{1})$ and so we can use Lemma \ref{lemma:codim1}. 

\begin{lemma}  \label{lemma:preimage}
The Jordan type of $x$ on $F_{1}^{\perp}/(\C[x]v+F_{1})$ is determined by the maximal $l$ such that $F_{1}^{\perp} \supseteq (x^{l-1})^{-1}(\C[x]v)$. It is then obtained $\Type(x,V/(\C[x]v+F_{1}))$ by removing the last box at the bottom of the $l$-th column. 
\end{lemma}
\begin{proof}
We first note that the condition $F_{1}^{\perp} \supseteq (x^{l-1})^{-1}(\C[x]v)$ is equivalent to $F_{1} \subseteq x^{l-1}((\C[x]v)^{\perp})$ and since $F_1 \subseteq \im(x^{\lambda_m-1})$ but $F_1 \not\subseteq \im(x^{\lambda_m})$, we have $l-1 < \lambda_m$. 
By Lemma \ref{lemma:codim1} we require the maximal $l$ such that  
$$
F_{1}^{\perp}/(\C[x]v+F_{1}) \supseteq \im(x_{|V/(\C[x]v+F_{1})})+\ker(x_{|V/(\C[x]v+F_{1})}^{l-1}).
$$ 
Translating this back to conditions on $F_1$ we calculate, 
\begin{eqnarray*}
\im(x_{|V/(\C[x]v+F_{1})}) 		&=& \im(x)+(\C[x]v+F_1) \quad \text{and}  \\
\ker(x_{|V/(\C[x]v+F_{1})}^{l-1})	&=& (x^{l-1})^{-1}(\C[x]v+F_1).
\end{eqnarray*} 
Thus we require $F_{1}^{\perp} \supseteq \im(x)+(\C[x]v+F_1) + (x^{l-1})^{-1}(\C[x]v+F_1)$. Since $F_{1}^{\perp}$ always contains $\im(x)$ and $\C[x]v+F_1$ is clearly contained in $(x^{l-1})^{-1}(\C[x]v+F_1)$, the essential requirement is:  
$ F_{1}^{\perp} \supseteq (x^{l-1})^{-1}(\C[x]v+F_1).$ But by Lemma \ref{lemma:preimagesum} and the fact that $(x^{l-1})^{-1}(F_1) \subseteq F_{1}^{\perp}$, this condition simplifies to $ F_{1}^{\perp} \supseteq (x^{l-1})^{-1}(\C[x]v).$\end{proof}

Recall that $m$ is defined to be the largest index such that $\alpha_i \neq 0$ or $\beta_i \neq 0$. 
Let $m''=\max\Delta_m+1$ be the smallest integer such that $m'' > m$ and $\nu_{m''} < \nu_{m}$. 
\begin{proposition} \label{prop:maxl}
The maximal $l$ such that $F_{1}^{\perp} \supseteq (x^{l-1})^{-1}(\C[x]v)$ is:
$$
l= \begin{cases} 
\mu_{m''}+\nu_m 			& 	\text{ if } 
\max\Gamma_m \leq \max\Delta_m\quad \text{and} \quad \displaystyle \sum_{ i \in \Delta_{m} : \,  i \leq m} \beta_i\neq 0, \\
\lambda_m				&	\text{otherwise.}
\end{cases}
$$
\end{proposition}
\begin{proof}
We first note that by definition of $m$ and Proposition \ref{prop:duplicate}, $F_1 \subseteq \im(x^{\lambda_m-1})$ but $F_1 \not\subseteq \im(x^{\lambda_m})$. Equivalently, we have $F_{1}^{\perp} \supseteq \ker(x^{\lambda_m-1})$ but $F_{1}^{\perp} \not\supseteq \ker(x^{\lambda_m})$. Since $\ker(x^{\lambda_m})$ is clearly contained in $(x^{\lambda_m})^{-1}(\C[x]v)$ we deduce that $l$ can be at most $\lambda_m$. \vspace{5pt}

Suppose first that $\max\Gamma_m > \max\Delta_m$. Then it is easy to see that the only vectors that have non-zero image in $\C[x]v$ under some power of $x$ (less than $\lambda_m$) are already contained in $\C[x]v$ up to some vector in the kernel of $x$ raised to that power. In this case we have $F_{1}^{\perp} \supseteq (x^{l-1})^{-1}(\C[x]v)$ if and only if  $F_{1}^{\perp} \supseteq \ker(x^{l-1})$; therefore $l=
\lambda_m$. So we suppose from now on that  $\max\Gamma_m \leq \max\Delta_m$. \vspace{5pt}

Suppose that $l < \lambda_m$. Then there must be a vector $w'$ in $(x^{l})^{-1}(\C[x]v) \setminus \ker(x^{l})$ which is not contained in $F_{1}^{\perp}$. Define: 
$$
w=\sum_{i=1}^{m''-1}v_{i,\mu_i+\nu_m}
$$
(noting that $m''-1=\max\Delta_m)$. Then it follows that $w' - w \in F_1^{\perp}$, by looking at the rightmost column of boxes on which the vector $w$ is supported. In this case we have $w$ is contained in  $(x^{\mu_{m''}+\nu_m})^{-1}(\C[x]v)$ since 
$$
x^{\mu_{m''}+\nu_m}(w)=\sum_{i=1}^{m''-1}v_{i,\mu_i-\mu_{m''}}=x^{\mu_{m''}}(v) \in \C[x]v,
$$ and we compute
$
\langle w,v_1 \rangle = \displaystyle \sum_{ i \in \Delta_{m}:~  i \leq m} \beta_i.
$
\\
\noindent
Therefore, $w$ is an element of $F_{1}^{\perp}$ if and only if $\sum_{\substack{ i \in \Delta_{m}  \\  i \leq m}} \beta_i = 0$. Thus if $\sum_{\substack{ i \in \Delta_{m}  \\  i \leq m}} \beta_i \neq 0$, then $F_{1}^{\perp} \supseteq (x^{\mu_{m''}+\nu_m-1})^{-1}(\C[x]v)$ but $F_{1}^{\perp} \not\supseteq (x^{\mu_{m''}+\nu_m})^{-1}(\C[x]v)$ and so the maximal $l$ in this case is $\mu_{m''}+\nu_m$. 
\end{proof}

\begin{remark}
Note that it is possible for $\mu_{m''}$ to be $0$ in Proposition \ref{prop:maxl}, for example this is the case when $\ell(\mu) < m''$. But the key point here is that $m''=\max\Delta_m+1$ and so the first value of $l$ in Proposition \ref{prop:maxl} can be written as $\mu_{\max\Delta_m+1}+\nu_{\max\Delta_m}$, which is indeed a part of the partition $\rho$. 
\end{remark}
We finish this section with two examples which illustrate the condition on the coefficients of $v_1$. 

\begin{example} \label{ex:right}
Consider the bipartition $(\mu, \nu)=((2^3,1),(2^2,1))$. 

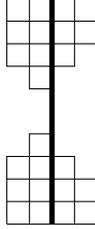
\begin{center}
\begin{tikzpicture}[scale=0.3]
\draw[-,line width=2pt] (0,5) to (0,-5);

\draw (-2,5) -- (-2,2) -- (0,2);
\draw (-2,5) -- (2,5) -- (2,3) -- (-2,3);
\draw (-2,4) -- (2,4);
\draw (-1,5) -- (-1,1) -- (0,1);
\draw (1,5) -- (1,2) -- (0,2);
\draw (-2,-5) -- (-2,-2) -- (0,-2);
\draw (-2,-5) -- (2,-5) -- (2,-3) -- (-2,-3);
\draw (-2,-4) -- (2,-4);
\draw (-1,-5) -- (-1,-1) -- (0,-1);
\draw (1,-5) -- (1,-2) -- (0,-2);

\end{tikzpicture}
\end{center}
Let $v_1=\alpha_1v_{11}+ \alpha_{2}v_{21}+\beta_{2}v_{24}^{*}+\beta_{1}v_{14}^{*}$ and put $F_1=\C{v_1}$. Here, the only vectors that have non-zero image in $\C[x]v$ are linear combinations of $u:=v_{13}+v_{23}+v_{33} (\in x^{-2}(\C[x]v))$ and vectors in $\C[x]v$. Now $u$ is clearly contained in $F_{1}^{\perp}$, so in this case, the maximal $l$ such that $F_{1}^{\perp} \supseteq (x^{l-1})^{-1}(\C[x]v)$ coincides with the maximal $l$ such that $F_{1}^{\perp} \supseteq \ker(x^{l-1})$. Therefore we can easily see that  $F_{1}^{\perp} \supseteq \ker(x^3)$ but $F_{1}^{\perp} \not\supseteq \ker(x^4)=V$. Hence the maximal $l$ in this case is $4$ and there is no restriction on the coefficients of $v_1$. \vspace{5pt}

Moreover, we can see that $F_{1} \subseteq \im(x^3) + \C[x]v$ but $F_{1} \not\subseteq \C[x]v$, and so the $k$ as in Proposition \ref{prop:maxk} is $4$ in this case. Therefore $\eType(v+F_1, x_{|F_{1}^{\perp}/F_{1}})=((2^3,1),(2,1^2))$ and is obtained by deleting the box from the bottom of the second column on the right of the wall. 
\end{example}

\begin{example}
By contrast to Example \ref{ex:right} where there were no conditions on the coefficients on $v_1$, consider the bipartition $((2^2,1),(2^2))$. 
\begin{center} \label{example:right}
\begin{tikzpicture}[scale=0.3]
\draw[-,line width=2pt] (0,4) to (0,-4);

\draw (-2,4) -- (-2,2) -- (0,2);
\draw (-2,4) -- (2,4) -- (2,2) -- (-2,2);
\draw (-2,3) -- (2,3);
\draw (-1,4) -- (-1,1) -- (0,1);
\draw (1,4) -- (1,2);
\draw (-2,-4) -- (-2,-2) -- (0,-2);
\draw (-2,-4) -- (2,-4) -- (2,-2) -- (-2,-2);
\draw (-2,-3) -- (2,-3);
\draw (-1,-4) -- (-1,-1) -- (0,-1);
\draw (1,-4) -- (1,-2);

\end{tikzpicture}
\end{center}

Again let $v_1=\alpha_1v_{11}+ \alpha_{2}v_{21}+\beta_{2}v_{24}^{*}+\beta_{1}v_{14}^{*}$. We first observe that $F_{1} \subseteq \im(x^3) + \C[x]v$, but $F_{1} \not\subseteq \C[x]v$ and so the $k$ as in Proposition \ref{prop:maxk} is again $k=4$. 
\vspace{5pt}

The key difference from Example \ref{ex:right} is that we do have vectors not already contained in $\C[x]v$ that have non-zero image in $\C[x]v$ - for example $x^{2}(v_{13}+v_{23})=x^{3}(v_{14}+v_{24})=v_{11}+v_{21}=x(v)$. Now $v_{14}+v_{24} \in F_{1}^{\perp}$ if and only if $\beta_1+\beta_2=0$ and so the maximal $l$ such that $F_{1}^{\perp} \subseteq (x^{l-1})^{-1}(\C[x]v)$ is $l=4$ if $\beta_1+\beta_2=0$ and $l=3$ if $\beta_1+\beta_2 \neq 0.$ We therefore have:
$$
\eType(v+F_1, x_{|F_{1}^{\perp}/F_{1}})= 
\begin{cases}
((2,1^2),(2^2)) & \text{if} \quad \beta_1+\beta_2=0, \\
((2^2,1),(2,1)) & \text{if} \quad \beta_1+\beta_2 \neq 0. \\
\end{cases}
$$
Here, the first case corresponds to deleting the box from the left hand side of the wall and the second case corresponds to deleting a box from the right hand side of the wall.

\end{example}

\section{Proof of Proposition \ref{prop:B-irred}}\label{sec:proof}

We can now put together all the results about Jordan types to prove Proposition \ref{prop:B-irred}. Fix $(v,x)\in\fN$, with $\eType(v,x)=(\mu,\nu)$. Let $F_1=\C v_1\in\P(\ker(x)\cap(\C[x]v)^\perp)$ such that $$\eType(v+F_1,x|_{F_1^\perp/F_1})=(\mu',\nu')$$ where $(\mu',\nu')$ is obtained from $(\mu,\nu)$ by removing a box in row $\overline{m}$, where $\lambda_{\overline{m}}>\lambda_{\overline{m}+1}$. We know by Proposition \ref{prop:duplicate} that we have
$$ v_1=\sum_{i\leq \ovm}\alpha_i v_{i,1}+\sum_{i\leq \ovm}\beta_iv_{i,\lambda_i}^*$$
where $\max\{i~|~\alpha_i\neq 0 \text{ or }\beta_i\neq 0\}=m$ with $\lambda_m=\lambda_{\overline{m}}$.
It follows that $\cB^{(\mu,\nu)}_{(\mu',\nu')}\subset \P^{2\overline{m}-1}$, where we take $[\alpha_1:\ldots:\alpha_{\overline{m}}:\beta_1:\ldots:\beta_{\overline{m}}]$ to be the homogeneous coordinates for $\P^{2\overline{m}-1}$. Using this inclusion, we will describe $\cB^{(\mu,\nu)}_{(\mu',\nu')}$ by giving conditions on the $\alpha_i$'s and $\beta_i$'s and therefore show that it is irreducible and has the required dimension. We examine different cases. 

\subsection{The case $\nu=\nu'$}\label{sec:mu} We assume now $\mu'$ is obtained by removing a box from row $\overline{m}$ in $\mu$, while $\nu=\nu'$. In particular this implies that $\mu_{\overline{m}}>\mu_{\overline{m}+1}\geq 0$ and that $\ovm= \max\Gamma_m\leq \max\Delta_m$.  \vspace{5pt}

If $\ovm>1$, by Theorem \ref{thm:T-AH} we then must have that 
\begin{align*}\sigma &= \Type(x,F_1^\perp /(\C[x]v+F_1) \\ &=(\mu'_1+\nu_1,\mu'_2+\nu_1,\mu'_2+\nu_2,\ldots)  \\
&=(\mu_1+\nu_1,\ldots,\mu_{\overline{m}-1}+\nu_{\overline{m}-1},\mu_{\overline{m}}-1+\nu_{\overline{m}-1},\mu_{\overline{m}}-1+\nu_{\overline{m}},\mu_{\overline{m}+1}+\nu_{\overline{m}},\ldots).  
\end{align*}
Comparing this to $\Type(x,V/\C[x]v)=\rho$, we notice that $\sigma$ is obtained from $\rho$ by reducing by $1$ the parts $\rho_{2\overline{m}-2}=\mu_{\overline{m}}+\nu_{\ovm-1}$ and $\rho_{2\ovm-1}=\la_{\ovm}=\mu_{\ovm}+\nu_{\ovm}$. \vspace{5pt}

Notice that if $\ovm=1$, then again by Theorem \ref{thm:T-AH} we get that $\sigma$ is obtained from $\rho$ by reducing by $1$ the part $\rho_1=\la_{1}=\mu_1+\nu_1$, while all the other parts stay the same. Observe that in either case
\begin{eqnarray}
b(\mu,\nu)-b(\mu',\nu) &=& 2N(\mu)+2N(\nu)+|\nu|-2N(\mu')-2N(\nu)-|\nu| \nonumber \\
&=&  2(N(\mu)-N(\mu'))=2(\ovm-1)=2\ovm-2. \label{eq:diff-b-mu}
\end{eqnarray}
We now divide this case into three subcases: $\ovm=1$, $\ovm>1$ and $\nu_{\ovm-1}>\nu_{\ovm}$, $\ovm>1$ and $\nu_{\ovm-1}=\nu_{\ovm}$. \vspace{5pt}

\subsubsection{$\ovm=1$}\label{subsub0} In this case, since $|\sigma|=|\rho|-1$, the only possibility is that $F_1\subset \C[x]v$, so that $\Type(x,V/(\C[x]v+F_1))=\Type(x,V/\C[x]v)=\rho$ and $\Type(x,F_1^\perp/(\C[x]v+F_1))$ is obtained from $\rho$ by removing a box in row $l=\la_1$, as in Proposition \ref{prop:maxl}. Since $F_1=\C v_1$ with $v_1=\alpha_1 v_{1,1}+\beta_1 v_{1,\la_1}^*$ and $v_1\in\C[x]v$, we have necessarily $\beta_1=0$ and $\alpha_1\neq 0$. It then follows from Proposition \ref{prop:maxl} that $l=\la_1$ without any other assumptions (this is because $\displaystyle \sum_{\substack{ i \in \Delta_{m}  \\  i \leq m}} \beta_i =\beta_1=0$). In this case, then we have
$$\cB^{(\mu,\nu)}_{(\mu',\nu)}=\{\C v_{1,1}\}\subset \P(\ker(x)\cap (\C[x]v)^\perp)$$
is a single point, which is indeed irreducible of dimension $2\ovm-2=2\cdot 1-2=0$. \vspace{5pt}

For the rest of Section \ref{sec:mu}, we assume that $\ovm>1$, which implies that $|\sigma|=|\rho|-2$, hence $F_1\not\subset \C[x]v$.
\subsubsection{$\nu_{\ovm-1}>\nu_{\ovm}$}\label{subsub1}
If $\nu_{\ovm-1}>\nu_{\ovm}$, then we have $\Lambda_m=\{\ovm\}$ (since $\mu_{\ovm}>\mu_{\ovm+1}$), and $m=\ovm$. Notice now that in Proposition \ref{prop:maxl} we cannot have $l=\mu_{m''}+\nu_m$ because $m''=\max\Delta_m+1\geq \ovm+1$, hence $\mu_{m''}+\nu_m\leq \mu_{\ovm +1}+\nu_m<\mu_{\ovm}+\nu_m$, and it then follows that the part of $\rho$ of size $l$ would not be $\rho_{2\ovm-2}$ nor $\rho_{2\ovm-1}$, which is impossible. Hence we have to have $l=\la_{\ovm}=\rho_{2\ovm-1}$, which implies by Proposition \ref{prop:maxl} that either $\nu_{\ovm}=0$ (in which case $\beta_{\ovm}=0$ by Remark \ref{rem:perp} ) or $\nu_{\ovm}\neq 0$ and $ \displaystyle \sum_{i \in \Delta_{m}  :~  i \leq m} \beta_i = 0$, which is the same as $\beta_{\ovm}=0.$
Since $\beta_{\ovm}=0$, we have $\alpha_{\ovm}\neq 0$, hence $\Gamma_m^\alpha\neq\varnothing\neq\Gamma_m^\beta$ and in Proposition \ref{prop:maxk} we obtain $k=\mu_m+\nu_{m'-1}$. Since the only possibility is that $k=\rho_{2\ovm-2}=\mu_{\ovm}+\nu_{\ovm -1}$, we get that $\nu_{m'-1}=\nu_{\ovm-1}$. Clearly $m'\leq \ovm$ because $\beta_{\ovm}=0$ and $\alpha_{\ovm}\neq 0$, thus $\nu_{m'-1}\geq \nu_{\ovm-1}$. Now, let $X\subset \P^{2\ovm-1}$ be the set of possible choices of the $v_1$ such that $\nu_{m'-1}>\nu_{\ovm-1}$. Then in terms of the homogeneous coordinates, $X$ is defined by the equations 
$$\alpha_{\ovm}=\alpha_{\ovm-1}=\ldots =\alpha_{m'}\text{ and }0=\beta_{\ovm}=\beta_{\ovm-1}=\ldots=\beta_{m'}$$
with $m'\leq\min\Delta_{\ovm-1}$. It is clear that this forms a closed subset of positive codimension of the set of all the $v_1$ such that $\beta_{\ovm}=0$ and $\alpha_{\ovm}\neq 0$. In conclusion, $\cB^{(\mu,\nu)}_{(\mu',\nu)}$ is an open subset of the set $Y=\{ \beta_{\ovm}=0;~\alpha_{\ovm}\neq 0\}\subset \P^{2\ovm-1}$. Since $Y$ is irreducible of dimension $2\ovm-2$, so is $\cB^{(\mu,\nu)}_{(\mu',\nu)}$.

\subsubsection{$\nu_{\ovm-1}\label{subsub2}=\nu_{\ovm}$}
In this case $\rho_{2\ovm-2}=\rho_{2\ovm-1}=\lambda_{\ovm}$, therefore, as in \ref{subsub1}, we have $l=\lambda_{\ovm}$. By Proposition \ref{prop:maxl}, there are then two possibilities. If $\nu_{\ovm}=\nu_{\ovm-1}=0$, we have from Remark \ref{rem:perp} that 
$ \displaystyle \sum_{i :~ \nu_i=0,~\mu_i>0} \beta_i= 0$. Otherwise, if $\nu_{\ovm}>0$, we have $\displaystyle \sum_{i \in \Delta_{m}:~  i \leq m} \beta_i = 0$. Now, our only choice for $k$ is also $k=\la_{\ovm}$. By Proposition \ref{prop:maxk} this is true when $\alpha_{\ovm}=0$ or $\beta_{\ovm}\neq 0$. If we are in the case $\alpha_{\ovm}\neq 0$ and $\beta_{\ovm}=0$, we have $\Gamma_m^\alpha\neq\varnothing\neq\Gamma_m^\beta$, hence $k=\mu_m+\nu_{m'-1}=\mu_{\ovm}+\nu_{\ovm}$, from which necessarily $\nu_{m'-1}=\nu_{\ovm}$. As in \ref{subsub1}, the set $X\subset \P^{2\ovm-1}$ where $\nu_{m'-1}>\nu_{\ovm}$ is a closed subset of positive codimension of the set $Y=\{\C v_1\in\P^{2\ovm-1}~|~\sum_{i\in I}\beta_i=0\}$ where 
$$ 
I=\begin{cases} \{i\in\Delta_m~|~i\leq m\} & \text{ if }\nu_{\ovm}>0, \\ 
\{i~|~\mu_i>0,~\nu_i=0\} & \text{ if }\nu_{\ovm}=0.
\end{cases}
$$
We also consider the set $Z=\{\C v_1\in Y~|~\alpha_i=0=\beta_i,~\forall i \in\Lambda_m\}$ which is also closed of positive codimension in $Y$. In conclusion, we have  
$$Y\setminus (X\cup Z)\subset \cB^{(\mu,\nu)}_{(\mu',\nu)}\subset Y\simeq \P^{2\ovm-2},$$ from which $\cB^{(\mu,\nu)}_{(\mu',\nu)}$ is indeed irreducible of dimension $2\ovm-2$.

\subsection{The case $\mu=\mu'$} We assume now $\nu'$ is obtained by removing a box from row $\overline{m}$ in $\nu$, while $\mu=\mu'$. In particular this implies that $\nu_{\overline{m}}>\nu_{\overline{m}+1}\geq 0$ and that $\ovm= \max\Delta_m\leq \max\Gamma_m$. By Theorem \ref{thm:T-AH} we then must have that 
\begin{align*}\sigma &= \Type(x,F_1^\perp /(\C[x]v+F_1) \\ &=(\mu_1+\nu'_1,\mu_2+\nu'_1,\mu_2+\nu'_2,\ldots)  \\
&=(\mu_1+\nu_1,\ldots,\mu_{\overline{m}}+\nu_{\overline{m}-1},\mu_{\overline{m}}+\nu_{\overline{m}}-1,\mu_{\overline{m}+1}+\nu_{\overline{m}}-1,\mu_{\ovm+1}+\nu_{\ovm+1},\ldots).  
\end{align*}
Comparing this to $\Type(x,V/\C[x]v)=\rho$, we notice that $\sigma$ is obtained from $\rho$ by reducing by $1$ the parts $\rho_{2\ovm-1}=\la_{\ovm}=\mu_{\ovm}+\nu_{\ovm}$ and $\rho_{2\overline{m}}=\mu_{\overline{m}+1}+\nu_{\ovm}$. In particular, since $|\sigma|=|\rho|-2$, we have $F_1\not\subset\C[x]v$. Observe that in this case
\begin{eqnarray}
b(\mu,\nu)-b(\mu,\nu') &=&2N(\mu)+2N(\nu)+|\nu|-2N(\mu)-2N(\nu')-|\nu'| \nonumber \\
 &=&2(N(\nu)-N(\nu'))+1 \nonumber\\
 &=&2\ovm-1.  \label{eq:diff-b-mu}
\end{eqnarray}
We now divide this case into two subcases: $\mu_{\ovm}>\mu_{\ovm+1}$ and $\mu_{\ovm}=\mu_{\ovm+1}$. 

\subsubsection{$\mu_{\ovm}=\mu_{\ovm+1}$}In this case we have that $\max\Gamma_m\geq \ovm +1>\max \Delta_m$. Notice that this implies that $\Gamma_m^\alpha=\varnothing$ because $\alpha_{\max\Gamma_m}=\alpha_{\ovm+1}=0$, hence by Proposition \ref{prop:maxk}, we have $k=\la_{\ovm}$. Also, since either $\mu_{\ovm}=0$ or $\max\Gamma_m>\max\Delta_m$, by Proposition \ref{prop:maxl} we get $l=\la_{\ovm}$. Notice that in this situation, $\rho_{2\ovm-1}=\rho_{2\ovm}=\la_{\ovm}=k=l$. If, as in \ref{subsub2}, we take $$Z=\{\C v_1\in\P^{2\ovm-1}~|~\alpha_i=0=\beta_i, \, \, \forall i \in\Lambda_m\},$$ which is a closed subset of positive codimension of $\P^{2\ovm-1}$, we then get 
$ \cB^{(\mu,\nu)}_{(\mu,\nu')}=\P^{2\ovm-1}\setminus Z$ is indeed irreducible of dimension $2\ovm-1$.

\subsubsection{$\mu_{\ovm}>\mu_{\ovm+1}$}

In this case we have that $\ovm=\max\Gamma_m=\max\Delta_m$ and $m''=\max\Delta_m+1=\ovm+1$. Notice that since $\rho_{2\ovm-1}=\lambda_{\ovm}>\rho_{2\ovm}=\mu_{\ovm+1}+\nu_{\ovm}$, we have to have 
$$
k=\rho_{2\ovm-1}=\lambda_{\ovm} \quad \text{and} \quad   l=\rho_{2\ovm}=\mu_{m''}+\nu_{\ovm}.
$$ 
The condition on $l$, by Proposition \ref{prop:maxl}, implies that $\displaystyle \sum_{ i \in \Delta_{m}:~  i \leq m} \beta_i \neq 0$. Using Proposition \ref{prop:maxk}, the condition on $k$ implies that one of the following is true: $\Gamma_m^\alpha=\varnothing$ (i.e. $\alpha_{\ovm}=0$), or $\Gamma_m^{\beta}=\varnothing$ (i.e. $\beta_{\ovm}\neq 0$), or $\Gamma_m^\alpha\neq\varnothing\neq\Gamma_m^{\beta}$ and $\nu_{m'-1}=\nu_{\ovm}$. Now, if we let 
$$
Y=\bigg\{\C v_1\in\P^{2\ovm-1} \bigg| \sum_{ i \in \Delta_{m}:~  i \leq m} \beta_i \neq 0,\text{ and }\beta_{\ovm}\neq 0 \bigg\},
$$ then $Y$ is an open irreducible subset of $\P^{2\ovm-1}$. Since we have $Y\subset  \cB^{(\mu,\nu)}_{(\mu,\nu')}\subset \P^{2\ovm-1}$, we obtain again that $\cB^{(\mu,\nu)}_{(\mu,\nu')}$ is irreducible of dimension $2\ovm-1$. \vspace{5pt}

This concludes the proof of Proposition \ref{prop:B-irred} and hence of Proposition \ref{prop:nested}. The following inequality is an immediate consequence: 
\begin{eqnarray} \label{eqn:weakbound}
\dim \cC_{(v,x)} \geq b(\mu,\nu). 
\end{eqnarray}

To conclude the proof of the main theorem, we will need to adapt a central object in Springer representations - the Steinberg Variety - to our settings.

\section{The Exotic Steinberg Variety and the Main Theorem} \label{section:steinbergvariety}
In this section we define the Exotic Steinberg Variety and present some of its properties that will enable us to complete the proof of Theorem \ref{mainthm}. First we give some background on the classical Steinberg Variety. \vspace{5pt}

Let $G=\Gl_n(\C)$, with Borel subgroup $B$ of upper triangular matrices, and let $\cN=\cN(\gl_n)$ be the nilpotent cone, the variety of all nilpotent $n \times n$ matrices. We have the Springer resolution of singularities
\begin{align*} & \pi: \widetilde{\cN}\twoheadrightarrow \cN;\qquad (F_\bullet,x)\mapsto x \\ \widetilde{\cN}&=\{(F_\bullet,x)\in G/B\times \cN ~|~ x(F_i)\subset F_{i-1},~\forall i=1,\ldots,n\}. \end{align*}
Then, we define the classical Steinberg variety
$$
 Z=\widetilde{\cN}\times _{\cN}\widetilde{\cN}=\{(F,F',x)\in G/B\times G/B\times \cN~|~ F,F'\in \pi^{-1}(x)\}.
 $$
Following \cite{Spa} and \cite{Ste}, the irreducible components of $Z$ can be parametrized in two ways. The first is by taking closures of the pre-images of the $G$-orbits on $G/B\times G/B$ under the projection $Z \twoheadrightarrow G/B \times G/B$; the parametrizing set is the Weyl group $S_n$ by the Bruhat lemma. The second way is as follows: given two standard tableaux $T, T'$ of shape $\lambda$ define $Z_{T, T'}$ be the set of all triples with $x$ a nilpotent of Jordan type $\lambda$,  $(\Type(x, F_i))_{i=1}^n = T$, $(\Type(x, F'))_{i=1}^n = T'$ (notice that the Jordan types of $x$ restricted to the spaces in the flag give an increasing sequence of partitions, which is the same thing as a standard tableau). Then the closures $\overline{Z_{T, T'}}$ comprise the irreducible components, and Spaltenstein and Steinberg independently showed that passing from one parametrization to the other is exactly given by the Robinson-Schensted Correspondence. We will construct an Exotic Robinson-Schensted Correspondence in Section \ref{sec:exoRS}. \vspace{5pt}

Let us now describe the exotic analogue of this construction.  We note here that all the techniques used below are classical and date back to Lusztig, Spaltenstein and Steinberg (and were explained to us by Anthony Henderson). Lemmas \ref{irrcomp}, \ref{dim} and \ref{irrcomp2} can  be found in \cite{kat}; Lemma \ref{equidimensional} can be found in \cite{kat5}. The proofs are very short so we include them for the reader's convenience. \vspace{5pt}

From Section \ref{sec:notation}, recall the symplectic flag variety $\cF(V)$, which using the embedding $\Sp_{2n}(\C)\hookrightarrow \Gl_{2n}(\C)$, can be realised as $\Sp_{2n}(\C)/(B \cap \Sp_{2n}(\C))$ (with $B$ a Borel subgroup of $\Gl_{2n}(\C)$).  Let $R^{+}$ be the set of positive roots of $Sp_{2n}(\C)$. Recall the resolution map $\pi:\widetilde{\fN}\to \fN$ of \eqref{eq:resol-nilcone} with fibres $\mathcal{C}_{(v, x)}$ over the point $(v,x)$.

\begin{definition}[(Exotic Steinberg Variety)]
The \emph{exotic Steinberg variety} is
$$ \fZ:=\widetilde{\fN}\times_{\fN}\widetilde{\fN}:=\{(F_\bullet,F'_\bullet,(v,x))\in \cF(V)\times\cF(V)\times \fN~|~F_\bullet,F'_\bullet\in\mathcal{C}_{(v,x)}\}.$$
\end{definition}

The Bruhat lemma says that the orbits of the $\Sp_{2n}$-action on $\cF(V)\times \cF(V)$ are in bijection with the Type C Weyl group $W(C_n)$. Using the embedding $\Sp_{2n}(\C)\hookrightarrow \Gl_{2n}(\C)$, we can identify $W(C_n)$ with a subset of the symmetric group $S_{2n}$:
$$ W(C_n)=\{w\in S_{2n}~|~w(2n+1-i)=2n+1-w(i)\}.$$
Note that this is a slightly different convention from the one used in Section \ref{sec:exoRS}.\vspace{5pt}

Given $w\in W(C_n)$, denote by $\cO_w$ the corresponding orbit in $\cF(V)\times\cF(V)$. Let $\theta: \fZ \longrightarrow \cF(V)\times\cF(V)$ and set $\fZ_{w}:=\theta^{-1}(\cO_{w})$. \\

\begin{definition}[(Flags in Relative Position)] \label{relative} Given two flags $F_{\bullet}, G_{\bullet} \in \cF(V)$, we say that $w(F_{\bullet}, G_{\bullet}) = w$, that is the two flags are in relative position $w \in W(C_n)$ if there is a basis $\{ v_{1}, \ldots, v_{2n} \}$ such that $\langle v_{i},v_{j} \rangle= \delta_{i+j, 2n+1}$ and for all $1 \leq i,j \leq 2n$, 
\begin{eqnarray*}
F_{i} &=& \C \{ v_1, \cdots, v_{i}\},   
\\ 
G_{j} &=& \C\{v_{w(1)}, \ldots, v_{w(j)}\}. 
\end{eqnarray*}

It is immediate that $w(F_\bullet,G_\bullet)=w$ if and only if $(F_\bullet,G_\bullet)\in\cO_w$.
\end{definition}

We can also parametrise the irreducible components of the $\fZ$ in two different ways. The first by elements of $W(C_n)$ and the second by pairs of bitableaux. We demonstrate the first parametrisation below and defer the second to Lemma \ref{irrcomp2}.

\begin{lemma} \label{irrcomp} 
The irreducible components of $\mathfrak{Z}$ are $\overline{\mathfrak{Z}_w}$ for $w \in W(C_n)$, and they all have the same dimension as $\mathfrak{N}$. 
\end{lemma} 
\begin{proof} 
It is clear that the union of $\mathfrak{Z}_w$, as $w$ runs over $W(C_n)$, is the entirety of $\mathfrak{Z}$. To show that they are the irreducible components, it suffices to check that they all have the same dimension, and that they are irreducible. We do this by analysing the restriction of $\theta$, namely $\theta_{w}: \mathfrak{Z}_w \rightarrow \mathcal{O}_w \subset \cF(V) \times \cF(V)$, which is a fibre bundle. \vspace{5pt}

For each $(F_{\bullet}, G_{\bullet}) \in \mathcal{O}_w$, choose a basis $\{ v_1, \ldots, v_{2n}\}$ as in Definition \ref{relative}, and a basis $E_{ij}$ of $\text{End}(\mathbb{C}^{2n})$ with $E_{ij} v_k = \delta_{jk} v_i$. Define the following matrices:
$$
\overline{E}_{ij}=
\begin{cases}
 E_{ij} + E_{2n+1-j, 2n+1-i} & \text{if} \quad 1 \leq i,j \leq n \\
E_{ij} - E_{2n+1-j, 2n+1-i}   & \text{otherwise}. 
\end{cases}
$$
Then the matrices $\{ \overline{E}_{ij} \; | \; 1 \leq i \leq 2n,~ 1 \leq j <2n+1-i \}$ constitute a basis for $\mathcal{S}$. Now for a fixed $w$ define: 
\begin{align*} 
S(w) &= \{ (i, j) \; | \; 1 \leq i < j < 2n+1-i,~ w^{-1}(i) < w^{-1}(j) \} \\ 
N(w) &= \text{span}\{ \overline{E_{ij}} \; | \; (i, j) \in S(w) \}. 
\end{align*} 

Below note that $x F_i \subset F_{i-1}$, $x G_i \subset G_{i-1}$ if and only if $x \in N(w)$. Now we have
\begin{align*} 
\theta_{w}^{-1}(F_{\bullet}, G_{\bullet}) 
=& \; \{ (v,x) \in \mathfrak{N} \; | \; v \in F_n \cap G_n,~ x F_{i} \subset F_{i-1}, ~x G_i \subset G_{i-1} \} \\ 
= & \; \{ (v,x) \in \mathfrak{N} \; | \; v \in \mathbb{C}\{v_1, \cdots, v_n\} \cap \mathbb{C}\{ v_{w(1)}, \cdots, v_{w(n)} \}, ~x \in N(w) \}
\end{align*}
and so 
$$
\dim(\theta_{w}^{-1}(F_{\bullet}, G_{\bullet})) = |S(w)| + | \{ i \leq n \; | \; w(i) \leq n \}| = |R^{+}| - \ell(w). 
$$
The above equality follows using the formula for the length $\ell(w)$ of a word in $W(C_n)$ as the number of negative roots made positive by the action of $w$. 
A simple calculation shows that $\text{dim}(\mathcal{O}_w) = |R^{+}| + \ell(w)$  and so it follows that $\text{dim}(\mathfrak{Z}_w) = 2 |R^{+}|=2n^2$. Hence all the $\mathfrak{Z}_w$'s have the same dimension. Since the fibres and base are irreducible, $\mathfrak{Z}_w$ is irreducible. This completes the proof.
\end{proof}

\begin{lemma} \label{dim} 
We have $ \dim(\cC_{(v,x)}) = b(\mu, \nu).$
\end{lemma} 
\begin{proof} 
Let $\mathbb{O}_{(\mu, \nu)}$ be the $Sp_{2n}$-orbit of $(v,x)$. Consider the following subvariety  $\mathfrak{Z}_{\mu, \nu} := \phi^{-1}(\mathbb{O}_{(\mu, \nu)})$ of $\mathfrak{Z}$, where $\phi$ is the natural projection $\mathfrak{Z} \longrightarrow \mathfrak{N}$. We can realise $\mathfrak{Z}_{\mu, \nu}$ as a fibre bundle over $\mathbb{O}_{(\mu, \nu)}$, with fibres $\cC_{(v,x)} \times \cC_{(v,x)}$ and hence
$$
\dim \mathfrak{Z}_{\mu, \nu} = \text{dim}(\mathbb{O}_{(\mu, \nu)})+ 2 \text{dim}(\cC_{(v,x)}).
$$
Therefore we have 
$$
 \dim \mathbb{O}_{(\mu, \nu)}+ 2 \text{dim}(\cC_{(v,x)}) = \text{dim}(\mathfrak{Z}_{\mu, \nu}) \leq \dim (\mathfrak{Z}) = \dim(\mathfrak{N})
 $$
 and so 
 $$
 \dim(\cC_{(v,x)}) \leq \frac{\text{dim}(\mathfrak{N}) - \dim(\mathbb{O}_{(\mu, \nu)})}{2} = b(\mu, \nu). 
$$
Since we have already shown (see \eqref{eqn:weakbound}) that $\text{dim}(\cC_{(v,x)}) \geq b(\mu, \nu)$, the equality follows. \end{proof}

With this lemma, we can now prove that the irreducible components of the exotic Steinberg variety can be parametrised in another way. 

\begin{definition} \label{definition:irrcomp2}
For $S,T \in \SYB(\mu,\nu)$, two standard Young bitableaux, define a subvariety of $\mathfrak{Z}$ as follows:
$$ 
\mathfrak{Z}_{\mu, \nu}^{S, T} := \{ (F_{\bullet}, F'_{\bullet}, (v,x) ) \in \fZ \, | \, (v,x)\in\mathbb{O}_{(\mu,\nu)},~F_{\bullet} \in \Phi^{-1}(S), F'_{\bullet} \in \Phi^{-1}(T) \}. 
$$ 
\end{definition} 

\begin{lemma} \label{irrcomp2} 
The irreducible components of $\mathfrak{Z}$ can also be parametrised as $\overline{\mathfrak{Z}_{\mu, \nu}^{S,T}}$ for $S, T \in \text{SYB}(\mu, \nu)$ and $(\mu,\nu) \in \cQ_{n}$. 
 \end{lemma} 
\begin{proof} We have a fibre bundle map $\phi_{S,T}: \mathfrak{Z}_{\mu, \nu}^{S, T} \rightarrow \mathbb{O}_{(\mu, \nu)}$ with fibres isomorphic to $\Phi^{-1}(S) \times \Phi^{-1}(T)$, where $\Phi$ is as in Definition \ref{def:Phi}. It is clear that the varieties $\mathfrak{Z}_{\mu, \nu}^{S,T}$ are irreducible (since the fibres of $\phi_{S, T}$ are irreducible), and that the closures $\overline{\mathfrak{Z}_{\mu, \nu}^{S,T}}$ are distinct. By looking at the map $\phi_{S,T}$, combined with Proposition \ref{prop:nested}, it follows that they have the same dimension as the exotic Steinberg variety: $$2b(\mu, \nu) + \dim(\mathbb{O}_{(\mu, \nu)}) = \text{dim}(\mathfrak{N}).$$

It follows that $\overline{\mathfrak{Z}_{\mu, \nu}^{S,T}}$ are irreducible components of the exotic Steinberg variety. Notice that 
$$\displaystyle\sum_{(\mu, \nu) \in \cQ_{n}} |\SYB(\mu, \nu)|^2 = |W(C_n)|$$ 
because the set $\SYB(\mu,\nu)$ labels a basis for the irreducible representation of $W(C_n)$ corresponding to the bipartition $(\mu,\nu)$. From Lemma \ref{irrcomp} it follows that there are no other irreducible components which gives the result.  \end{proof}

\begin{lemma} \label{equidimensional} All irreducible components of $\mathcal{C}_{(v,x)}$ have maximal dimension. \end{lemma} 

\begin{proof} In \cite{Spa77}, Spaltenstein proves that the components of Springer fibres for reductive groups have the same dimension. The proof is elementary, and has been adapted to the exotic setting; see Theorem 7.1 and Section 8 of \cite{kat5}. 
We now explain the setup. Fix a flag $F^0_{\bullet} \in \mathcal{F}(V)$ 
and define a subvariety of the exotic nilpotent cone: 
$$ \mathbb{V}^+ = \{ (u,y)\in\mathfrak{N} \; | \; u \in F^0_n; \; y (F^0_i) \subseteq F^0_{i-1} \text{ for all } i \}=\{ (u,y)\in\mathfrak{N}~|~F^0_\bullet\in\cC_{(u,y)}\}. $$

We define two maps 
$$\pi_1:\Sp_{2n} \longrightarrow \mO_{(\mu,\nu)} \quad \text{and} \quad \pi_2: \Sp_{2n} \longrightarrow \mathcal{F}(V)$$ 
given by 
$$\pi_1(g)=(g^{-1}v, g^{-1}xg) \qquad \pi_2(g)= g\cdot F^0_{\bullet}.$$
These are both fibre bundles: the fibres of $\pi_1$ are isomorphic to the stabilizer $Z$ of $(v,x)$ inside $\Sp_{2n}$, while the fibres of $\pi_2$ are Borel subgroups. 
Let $Y =\pi_1^{-1}(\mO_{\mu,\nu} \cap \mathbb{V}^+)= \pi_2^{-1}(\mathcal{C}_{(v,x)})$.
Since both $Z$ and Borel subgroups are irreducible, the preimages under $\pi_1$ (resp. $\pi_2$) of the irreducible components of $\mO_{\mu,\nu} \cap \mathbb{V}^+$ (resp. $\cC_{(v,x)}$) are the irreducible components of $Y$.
It follows that there is a bijection between the irreducible components $\{C_i~|~i\in I\}$ of  $\cC_{(v,x)}$ and the irreducible components $\{O_i~|~i\in I\}$ of $\mO_{\mu,\nu} \cap \mathbb{V}^+$. Hence, it also follows that, for each $i\in I$, 
\begin{equation}\label{eq:dimequal} 
\dim(C_i)+\dim(Z)=\dim(O_i)+\dim(B) 
\end{equation} 
where $B\subset \Sp_{2n}$ is a Borel subgroup. \vspace{5pt}

By Theorem 7.1 of \cite{kat5}, all the irreducible components of $\mO_{\mu,\nu} \cap \mathbb{V}^+$ (the notation used in \cite{kat5} is $\mathcal{O}\cap\mathbb{V}^+$ instead) have the same dimension, and so the result then follows from \eqref{eq:dimequal}. 
\end{proof}

We can now prove the main theorem. 

\begin{proof}[of Theorem \ref{mainthm}] Using Lemma \ref{dim}, we know that $\{ \overline{\Phi^{-1}(T)} \, | \, T \in \SYB(\mu,\nu) \}$ form a non-redundant set of irreducible components of $\cC_{(v,x)}$. It suffices to show that there are no other irreducible components. \vspace{5pt}

Suppose $C$ is another irreducible component of $\mathcal{C}_{(v,x)}$, and let 
$$C'=C\setminus\bigcup_{T \in \SYB(\mu,\nu)}\Phi^{-1}(T).$$
Then $C'$ is an irreducible subvariety of $\cC_{(v,x)}$ of maximal dimension by Lemma \ref{equidimensional}. 
Consider the subvariety of $\mathfrak{Z}$: 
$$
\mathfrak{Z}_{\mu, \nu}^{C,C} = \{ g\cdot(F_{\bullet}, G_{\bullet}, (v,x)) \; | \; F_{\bullet}, G_{\bullet} \in C', g \in Sp_{2n} \}.
$$

Since the stabilizer of $(v,x)$ in $\Sp_{2n}$ is connected, $\mathfrak{Z}_{\mu,\nu}^{C,C}$ is a fibre bundle over $\mathbb{O}_{(\mu, \nu)} $ with fibres isomorphic to $C'\times C'$ (both $C$ and the $\Phi^{-1}(T)$'s are left invariant by the stabilizer). Using Lemma \ref{equidimensional}, and the same argument from Lemma \ref{irrcomp2}, $\mathfrak{Z}_{\mu, \nu}^{C,C}$ is an irreducible subvariety of $\mathfrak{Z}$ of maximal dimension. Hence it must be contained in $\overline{\mathfrak{Z}_{\mu, \nu}^{S,T}}$ for some $S$ and $T$. This is a contradiction, since $\mathfrak{Z}_{\mu, \nu}^{C,C}$ and $\mathfrak{Z}_{\mu, \nu}^{S,T}$ are disjoint. 
\end{proof}

\section{Exotic Type C Robinson-Schensted Correspondence} \label{sec:exoRS}

As mentioned in Section \ref{section:steinbergvariety}, the classical Robinson-Schensted correspondence, which is a bijection between the symmetric group $S_n$ and pairs of standard Young tableaux of the same shape, was rediscovered in the geometry of the Steinberg variety due to Spaltenstein and Steinberg (\cite{Spa}, \cite{Ste}). In this section, we describe how the same techniques can be used to give a variant of the Robinson-Schensted correspondence in type C using the geometry of the exotic nilpotent cone.  In \cite{NRS17}, the sequel to this paper, we give an explicit combinatorial description of this bijection. \vspace{5pt}

We have shown in Lemmas \ref{irrcomp} and \ref{irrcomp2} that the irreducible components of $\fZ$ are parametrised in two different ways: one by elements of the Weyl group $W(C_n)$ and the other by pairs of standard Young bitableaux $(T,T')\in \SYB(\mu, \nu)$ as $(\mu,\nu)$ runs over all bipartitions of $n$.

By comparing the two parameterisations, we deduce the existence of an {\it Exotic Robinson-Schensted Correspondence}:

\begin{corollary} \label{cor:rsk} 
There is a bijection: 
$$ W(C_n) \longleftrightarrow \coprod_{(\mu,\nu)\in\cQ_n} \SYB(\mu,\nu) \times \SYB(\mu,\nu) $$
defined geometrically by $w \leftrightarrow (T,T')$ if and only if $\overline{\fZ_w}=\overline{\fZ_{(\mu,\nu)}^{T,T'}}$. 
\end{corollary}

It is tempting to conjecture that the geometric correspondence of Corollary \ref{cor:rsk} should be given by a naive Type C version of the Robinson-Schensted algorithm, but in fact this is not the case (as we will see in the $n=2$ example below). See Section 3 of \cite{NRS17} for a complete combinatorial description of this bijection.

\subsection{The $n=2$ case}

If we identify the Weyl group $W(C_n)$ with  $\Z_2 \wr  S_n$, we can write its elements as permutation words in the letters $1,2,\ldots, n$ with some of the letters barred. The `naive' version of the type C Robinson-Schensted Correspondence uses the usual row bumping algorithm with the only difference being that the numbers with bars will be inserted in the second tableau of the bitableau and the number without bars will be inserted in the first tableau of the bitableau; see the below example for an illustration:
 $$
 2\bar{6}\bar{3}1854\bar{7} \longleftrightarrow \left(\left(\young(41,52,:8),\young(37,6)\right),\left(\young(51,64,:7), \young(28,3)\right)\right)
 $$

We now describe the exotic Robinson-Schensted correspondence from Corollary \ref{cor:rsk} in the case of $n=2$, to point out that it is indeed different from the `naive' Type C case that we just explained. \vspace{5pt}

We let $s=\bar{1}2$ and $t=21$ be generators for the Weyl group $W(C_2)$. 

\begin{eqnarray*}
W(C_2) 			&\longleftrightarrow&  \coprod_{(\mu,\nu)\in\cQ_2}\SYB(\mu,\nu)\times\SYB(\mu,\nu)				\\
\text{id}=12 		&\longleftrightarrow& \left(\left(\young(21); - \right),\left(\young(21); - \right)\right) 				\\
s=\bar{1}2			&\longleftrightarrow& \left(\left(\young(2); \young(1) \right),\left(\young(2); \young(1)\right)\right)  	\\
t=21				&\longleftrightarrow& \left(\left(\young(1); \young(2) \right),\left(\young(1); \young(2)\right)\right)   	\\
ts=\bar{2}1		&\longleftrightarrow& \left(\left(\young(1); \young(2) \right),\left(\young(2); \young(1)\right)\right) 		\\  
st=2\bar{1}		&\longleftrightarrow& \left(\left(\young(2); \young(1) \right),\left(\young(1); \young(2)\right)\right) 		\\  
sts=\bar{2}\bar{1}	&\longleftrightarrow& \left(\left(-; \young(12) \right),\left(-; \young(12)\right)\right)			 		\\  
tst=1\bar{2}		&\longleftrightarrow& \left(\left(\young(1,2); - \right),\left(\young(1,2); - \right)\right) 		 		\\  
stst=\bar{1}\bar{2}	&\longleftrightarrow& \left(\left(-; \young(1,2) \right),\left(-; \young(1,2)\right)\right)	
\end{eqnarray*}

We now illustrate our method for computing this geometric correspondence with a couple of examples. We will fix a point in the orbit of the exotic nilpotent given by a certain bipartition $(\mu,\nu)$ and, for two given bitableaux $T$, $T'$ we will find flags that are generic within the corresponding exotic Springer fibres. We will then obtain an element of the Weyl group by looking at the relative position of the two flags (i.e. the $\Sp_{2n}$-orbit passing through them). To compute the relative position of the two flags, we compute dimensions of the intersections of the spaces in the flags. \vspace{5pt}

For our purposes, it will help us to identify the Weyl group $W(C_n)$ as signed permutations of the set $\{1,\ldots ,n\} \cup \{\bar{1}, \ldots, \bar{n}\}$. For example, if $n=3$ and $w=\bar{2}13$ then we understand this as $w(1)=\bar{2}, w(2)=1$ and $w(3)=3$, and so $w(\bar{1})=2, w(\bar{2})=\bar{1}$ and $w(\bar{3})=\bar{3}$. For a general $n$, we define $s_0=\bar{1}2\ldots n$ and $s_i$ to be the standard generators for $S_n$ for $1 \leq i \leq n$. Thus a presentation for $W(C_n)$ with these generators is given by: 
$$
\langle s_0, s_1, \ldots, s_{n-1} \, | \, s_{i}^2=1, (s_{0}s_{1})^{4}=1, (s_{i}s_{i+1})^3=1, \, \text{for} \, \, i \geq 1, (s_{i}s_{j})^2=1, \, \text{for} \,  |i-j| > 1 \rangle,
$$
and we embed this in $S_{2n}$ via the map $\iota$ as follows:
$$ 
\iota(s_0)=(n,n+1) \quad \text{and} \quad  \iota(s_i)=(n-i,n-i+1)(n+i,n+i+1) \quad \text{for} \, \, 1 \leq i \leq n-1,
$$
where we use cycle notation for the transpositions of $S_{2n}$. \vspace{5pt}

We realise the image of $\iota$ by permutation matrices, where the columns of the matrix left to right are labelled $n(n-1)\cdots 1\bar{1}\bar{2}\cdots \bar{n}$ and same with the rows from top to bottom. Then $w=w_1w_2\cdots w_n$ describes the permutation given by the first $n$ columns, and the other columns are obtained by symmetry.

\begin{example} 
Consider the generic orbit that corresponds to the tableau $(\young(21); -)$. Since the exotic Springer fibres consists of a single point in this case, the relative position between two generic points trivially corresponds to the identity permutation. Namely: 
$$
\text{id}=12 \longleftrightarrow \left(\left(\young(21); - \right),\left(\young(21); - \right)\right) 		
$$ 
\end{example}

\begin{example} \label{example:fibre1} 
Suppose $F_{\bullet}$ is a generic point in the fibre $\Phi^{-1}(\left(\young(1),\young(2)\right))$. In this case the space $V$ is pictured as follows:
\begin{center}
\begin{tikzpicture}[scale=0.5]
\draw[-,line width=2pt] (0,2) to (0,-2); 
\draw (-1,0.5) -- (1,0.5) -- (1,1.5) -- (-1,1.5) -- (-1,0.5);
\node[right] at (-1.1,1){\tiny{$v_{11}$}}; \node[left] at (1.1,1){\tiny{$v_{12}$}};
\draw (-1,-0.5) -- (1,-0.5) -- (1,-1.5) -- (-1,-1.5) -- (-1,-0.5);
 \node[right] at (-1.1,-1){\tiny{$v_{12}^{*}$}}; 
 \node[left] at (1.1,-1){\tiny{$v_{11}^{*}$}}; 
 \end{tikzpicture} 
 \end{center} 
 with $v=v_{11}$. A generic point in $\ker(x) \cap (\C[x]v)^{\perp}$ has the form $v_1=\alpha_{1}v_{11}+\beta_{1}v_{12}^{*}$ with $\alpha_1, \beta_1 \in \C$. Since in this case we must have $\eType(v+F_1, x_{|F_{1}^{\perp}/F_{1}})=\left(\yng(1);\varnothing\right)$, we require $\beta_1 \neq 0$. Since $v \in F_2$, a generic point in the fibre has the form: 
 \begin{eqnarray*}
  F_1 	&=& \C\{\alpha_{1}v_{11}+\beta_{1}v_{12}^{*}\}, \\ 
  F_2	&=& \C\{v_{11}, v_{12}^{*} \},  \\ 
  F_3	&=& \C\{v_{11}, v_{12}^{*}, \alpha_{1}v_{12} + \beta_{1}v_{11}^{*} \}. 
  \end{eqnarray*} 
  Let $F_{\bullet}'$ be another generic point in the fibre. We have the intersection and permutation matrices recording the relative position of $F_{\bullet}$ with respect to $F_{\bullet}'$: 

  \begin{center} \begin{tabular}{c|cccc}  		
  		& $F_1$	& $F_2$	& $F_3$ 	& $F_4$ 		\\ 
		\hline 
		$F_1'$ 	& $0$ 	& 	$1$	& 	$1$	& 	$1$		\\ 
		$F_2'$ 	& $1$	& 	$2$	& 	$2$	& 	$2$		 \\  
		$F_3'$ 	& $1$	& 	$2$	& 	$2$	& 	$3$ 		 \\ 
		$F_4'$ 	& $1$	& 	$2$	& 	$3$	& 	$4$		 \\ 
		\end{tabular} \quad $\leadsto$ \quad  
		\begin{tabular}{c|cccc}  				
		& $2$		& $1$	& $\bar{1}$ & $\bar{2}$ 				\\ 
		\hline 
		$2$ 		& 	 	& 	$1$	& 		& 			\\ 
		$1$ 		& 	$1$	& 		& 		& 			 \\   
		$\bar{1}$ 	& 		& 		& 		& 	 $1$	 \\ 
		$\bar{2}$ 	& 		& 		& 	$1$	& 			 \\ 
		\end{tabular}
\end{center} 
which yields the permutation $w(F_{\bullet}, F_{\bullet}') = 21$. The permutation matrix $(b_{ij})$ is obtained from the intersection matrix $(a_{ij})$ by the formula: $b_{ij}=a_{i-1,j-1}+a_{ij}-a_{i-1,j}-a_{i,j-1}$ for all $1 \leq i,j \leq 2n$. Therefore under the exotic Robinson-Schensted correspondence, we have:  
$$ 
21 \longleftrightarrow \left(\left(\young(1);\young(2)\right),\left(\young(1);\young(2)\right)\right). 
$$ 
The other cases can be computed in a similar way.
\end{example}

\section{Exotic Springer Fibres of Dimension Two}\label{sec:dimension2}

In this section we give a description of the irreducible components in the case where the exotic Springer fibres has dimension two. We show that, in these cases, the components are $\mathbb{P}^1$ bundles over $\mathbb{P}^1$, and classify the Hirzebruch surfaces. In \cite{L86}, Lorist gives a description of the irreducible components for Springer fibres with dimension two in types A, D and E; this section is an extension of his results to the exotic case. In \cite{F86} and \cite{SW12}, Fung and Stroppel-Webster give a combinatorial approach to the structure of a Springer fibre for a two-row nilpotent in type A. We expect that the same techniques can be used to give a combinatorial approach to exotic Springer fibres and the present section is a step in that direction (but we only look at the case where the exotic Springer fibre has dimension two).   \vspace{5pt}

We will use the map $\Phi$ from Definition \ref{def:Phi} and the dimension formula from Definition \ref{def:bdim}. Using the above mentioned dimension formula, if $b(\mu, \nu)=2$, then either:

\begin{enumerate} 
\item $|\nu| = 0$, and $\displaystyle \sum_{i \geq 1} (i-1) \lambda_i = 1$. The second condition implies that $\lambda_2 = 1$, $\lambda_i = 0$ for $i \geq 3$. Hence $\mu = (n-1, 1)$ and $\nu = 0$. \item $|\nu|=2$, and $ \displaystyle\sum_{i \geq 1} (i-1) \lambda_i = 0$. The second condition implies that $\lambda_i = 0$ for $i \geq 2$. Hence $\nu = (2)$, and $\mu = (n-2)$. 
\end{enumerate}

We remark that the case where $(\mu,\nu)$ is a pair of arbitrary one-row partitions will be dealt with in full generality in the forthcoming article \cite{SW}. \vspace{5pt}

For what follows we introduce the following notation. Suppose $F_{\bullet} \in \cC_{(v,x)}$, where $(v,x)$ has exotic type $(\mu,\nu)$. Then the flag $F_{\bullet}/F_1:=(0 \subseteq F_2/F_1 \subseteq \cdots \subseteq F_1^{\perp}/F_1)$ obtained from $F_{\bullet}$ is a point in the fibre $\cC_{(v+F_1 , F_{1}^{\perp}/F_{1})}$ corresponding to the degree $n-1$ exotic nilpotent cone.   \vspace{5pt}

\textbf{Case 1.} For $2\leq i\leq n$, let $T_i^n$ be the bitableau of shape $(\mu,\emptyset)$, where $\mu = (n-1,1)$ with $\{ 1,2, \cdots, i-1, i+1, \cdots, n \}$ in the first row, and $i$ in the second row. Let $C_{n}^{i} = \Phi^{-1}(T_n^i)$, and let its closure $\overline{C_n^i}$ be the irreducible component of $\cC_{(v,x)}$ corresponding to $T_i^n$.  \vspace{5pt}

\begin{proposition} \label{cni} We have $C_n^i \simeq \mathbb{P}^2 \backslash pt$. \end{proposition}
\begin{proof}Let $F_\bullet\in C_n^i$, then for $1\leq k\leq n-i$, we have 
$$\eType(v+F_k, x|_{F_k^{\perp}/F_k})=((n-k-1,1),\emptyset).$$
Hence, for $k=1$, this implies that $(x|_{F_1^{\perp}/F_1})^{n-2} (v+F_1) = 0$, which means that $x^{n-2}v\in F_1$, but since $x^{n-2}v\neq 0$, we have that $F_1=\C\{x^{n-2}v\}$. For $k>1$, recursively, we obtain that $x^{n-k-1}v\in F_k$, hence
$$ F_k=\C\{x^{n-2}v, \ldots, x^{n-k-1}v\}~\qquad \text{for} \quad 1\leq k\leq n-i.$$
Now, since $v\in F_n$, we have $x^j v\in F_{n-j}$, in particular $x^{i-2}v\in F_{n-i+2}$, but since
$$ 
\eType(v+F_{n-i+1}, x|_{F_{n-i+1}^{\perp}/F_{n-i+1}})=((i-1),\emptyset),
$$
we have that $(x|_{F_{n-i+1}^{\perp}/F_{n-i+1}})^{i-2}(v+F_{n-i+1})\neq 0$, i.e. $x^{i-2}v\not\in F_{n-i+1}$. \vspace{5pt}

For simplicity of notation, let $V^j=\C\{x^{n-2}v, \ldots, x^{j}v\}$, $j=i-1,i-2$. From the above it follows that
$$ F_{n-i+1}\subset V^{i-2}+\ker(x)\cap (\C v)^\perp,\qquad \text{but} \quad F_{n-i+1}\neq V^{i-2}.$$
For $n-i+1<k\leq n$, since $x^{n-k}v\in F_k\setminus F_{k-1}$, we have that $F_k=F_{k-1}\oplus \C\{x^{n-k}v\}$.
So, any flag $F_\bullet\in C_n^i$ is determined by the choice of $F_{n-i+1}$ and so the map 
$$
q: C_n^i \rightarrow \mathbb{P}\left( (V^{i-2}+ \ker(x)\cap (\C v)^\perp)/V^{i-1}\right) \setminus V^{i-2}
$$ 
which sends $q(F_{\bullet}) = F_{n-i+1}/V^{i-1}$, is an isomorphism of varieties and so the conclusion follows.
\end{proof}

\begin{proposition} \label{P1bundle} 
We have $\overline{C_n^i}$ is a $\mathbb{P}^1$-bundle over $\mathbb{P}^1$. 
\end{proposition} 
\begin{proof} 
Suppose that $F_{\bullet} \in C_n^i$; from the proof of Proposition \ref{cni}, using the same notation, we have that 
$$V^{i-2}=V^{i-1}\oplus \C\{x^{i-2}v\}\subset F_{n-i+1}\oplus \C\{x^{i-2}v\}\subset V^{i-2}+ \ker(x)\cap (\C v)^\perp.$$
Since $F_{n-i+1}\oplus \C\{x^{i-2}v\}=F_{n-i+2}$, it follows that 
\begin{equation}\label{eq:F2}
V^{i-2}\subset F_{n-i+2}  \subset V^{i-2}+ \ker(x)\cap (\C v)^\perp.
\end{equation}
We also have
\begin{align}\label{eq:Fk1}  
F_k=\C\{x^{n-2}v, \ldots, x^{n-k-1}v\} &  \qquad 1\leq k\leq n-i, \\
\label{eq:Fk} F_k=F_{k-1}\oplus \C\{x^{n-k}v\} & \qquad n-i+3\leq k\leq n.
\end{align}
These are closed conditions and so it follows that \eqref{eq:F2}-\eqref{eq:Fk} hold for any $F_{\bullet} \in \overline{C_n^i}$. \vspace{5pt}

Consider the map $\eta: \overline{C_n^i} \rightarrow X := \mathbb{P}((V^{i-2}+\ker(x) \cap (\mathbb{C}v)^{\perp}) /  V^{i-2})$, given by $\eta(F_{\bullet}) = F_{n-i+2}/V^{i-2}$. It suffices to show that it is surjective, and has $\mathbb{P}^1$-fibres. To see that the map $\eta$ is surjective, it suffices to note that using Proposition \ref{cni}, the map $\eta{|_{C_n^i}}$ is surjective.  \vspace{5pt}

Now suppose that we fix $F_{n-i+2}/V^{i-2}\in X$, and let $F_\bullet\in\eta^{-1}\left(F_{n-i+2}/V^{i-2}\right)$. From \eqref{eq:Fk1} and \eqref{eq:Fk}, all of the spaces of $F_\bullet$ are determined except for $F_{n-i+1}$, and $$F_{n-i+1}\in\mathbb{P}(F_{n-i+2}/V^{i-1})\simeq \mathbb{P}^1.$$\end{proof}

\begin{corollary} We have $\overline{C_n^i} \cap \overline{C_n^{i+1}} \simeq \mathbb{P}^1$; and if $j-i>1$ the intersection $\overline{C_n^i} \cap \overline{C_n^j}$ is empty.  \end{corollary}

\begin{proof} If $F_{\bullet} \in \overline{C_n^j}$, then from the proof of Proposition \ref{P1bundle}, $F_{n-j+2} \subseteq x^{-1}(F_{n-j})$. If we also have $F_{\bullet} \in \overline{C_n^i}$, with $j-i>1$, then $x^{j-3}v \in F_{n-j+2}$; since $x^{j-2}v \notin F_{n-j}$, it follows that $F_{n-j+2} \not \subset x^{-1}(F_{n-j})$. This is a contradiction, hence the intersection $\overline{C_n^i} \cap \overline{C_n^j}$ is empty. \vspace{5pt}

Now suppose that $F_\bullet\in \overline{C_n^i} \cap \overline{C_n^{i+1}}$, then by \eqref{eq:Fk1}, $F_k$ is determined for $1\leq k\leq n-i$. By applying \eqref{eq:F2} and \eqref{eq:Fk} to $\overline{C_n^{i+1}}$ we obtain that
\begin{equation}\label{eq:F1} V^{i-1}\subset F_{n-i+1}\subset V^{i-1}+\ker(x)\cap(\C v)^\perp \end{equation}
and $F_k$ is determined from $F_{k-1}$ for $n-i+2\leq k \leq n$.
In conclusion, any $F_\bullet\in \overline{C_n^i} \cap \overline{C_n^{i+1}}$ is determined by the choice of $F_{n-i+1}$, with
$$ F_{n-i+1}\in \mathbb{P}((V^{i-1}+\ker(x)\cap(\C v)^\perp)/V^{i-1})\simeq \mathbb{P}^1.$$ \end{proof}

\begin{definition} For a scheme $X$, and a finite-dimensional vector bundle $\vartheta: V \rightarrow X$, we define an associated
projective bundle $\tilde{\vartheta}: \mathbb{P}(V) \rightarrow X$ whose fibre $\tilde{\vartheta}^{-1}(x)$ over a base point $x \in X$ is just the projectivisation of the affine fibre $\vartheta^{-1}(x)$. For $n>0$, let $\Sigma_{-n}$ be the $\mathbb{P}^1$-bundle over $\mathbb{P}^1$ obtained as the associated projective bundle to the two-dimensional vector bundle $\mathcal{O} \oplus \mathcal{O}(-n)$ over $\mathbb{P}^1$. \end{definition} 

\begin{remark} \label{tensorL} Given a line bundle $\mathcal{L}$ on $X$, one may check that $\mathbb{P}(V) \simeq \mathbb{P}(\mathcal{L} \otimes V)$; in particular if $j<i$ the projective bundle associated to $\mathcal{O}(i) \oplus \mathcal{O}(j)$ is $\Sigma_{j-i}$. \end{remark}

\begin{proposition} \label{Sigma-1} We have an isomorphism, $\overline{C_n^i} \simeq \Sigma_{-1}$. \end{proposition}

\begin{proof} Because of \eqref{eq:F2}-\eqref{eq:Fk}, we have $\overline{C_n^i} \simeq \overline{C_i^i}$, hence it suffices to prove that $\overline{C_n^n} \simeq \Sigma_{-1}$. From Proposition \ref{P1bundle}, it follows that: 
$$ 
\overline{C_n^n} \simeq \{ (F_1, F_2) \; | \; \mathbb{C}\{x^{n-2}v\} \subset F_2 \subset \text{ker}(x) \cap (\mathbb{C}v)^{\perp}, F_1 \subset F_2 \}. 
$$ 
Define the variety $\widehat{C_n^n}$ as follows: 
$$
\widehat{C_n^n} \simeq \{ (F_2, s) \; | \; \mathbb{C}\{x^{n-2}v\} \subset F_2 \subset \text{ker}(x) \cap (\mathbb{C}v)^{\perp}, s \in F_2 \}. 
$$ 
We have a natural map $\widehat{C_n^n} \rightarrow \mathbb{P}(\text{ker}(x) \cap (\mathbb{C}v)^{\perp}/ \mathbb{C}\{x^{n-2}v\}) \simeq \mathbb{P}^1$. Under this map, $\widehat{C_n^n}$ is a two-dimensional vector bundle over $\mathbb{P}^1$, and its projectivisation is isomorphic to $\overline{C_n^n}$. Hence it suffices to show that $\widehat{C_n^n} \simeq \mathcal{O}_{\mathbb{P}^1} \oplus \mathcal{O}_{\mathbb{P}^1}(-1)$; this implies that $\overline{C_n^n} \simeq \Sigma_{-1}$. \vspace{5pt}

Consider the sub-bundle of $\widehat{C_n^n}$ consisting of pairs $(F_2, s)$ with $s \in \mathbb{C}\{x^{n-2}v\}$; it is isomorphic to $\mathcal{O}_{\mathbb{P}^1}$, and the quotient is isomorphic to $\mathcal{O}_{\mathbb{P}^1}(-1)$. Since there are not any non-trivial extensions between $\mathcal{O}_{\mathbb{P}^1}$ and $\mathcal{O}_{\mathbb{P}^1}(-1)$, the conclusion follows. \end{proof}

\begin{corollary} Recalling that $\Sigma_{-1}$ is isomorphic to the blow-up of $\mathbb{P}^2$ at a point $a \in \mathbb{P}^2$, there is a natural embedding $\mathbb{P}^2 \backslash \{a\} \rightarrow \Sigma_{-1}$. The map $C_n^i \rightarrow \overline{C_n^i}$ can be identified with this embedding; in particular, $\overline{C_n^i} \backslash C_n^i \simeq \mathbb{P}^1$. \end{corollary} 

\begin{remark} 
See page 4 of Ravi Vakil's notes, \cite{rvnotes} for a proof of the fact that $\Sigma_{-1}$ is isomorphic to the blow-up of $\mathbb{P}^2$ at a point $a \in \mathbb{P}^2$. 
\end{remark}

\begin{proof} It suffices to prove these statement for the inclusion $C_n^n \hookrightarrow \overline{C_n^n}$. After identifying $\overline{C_n^n}$, the space of $(F_1, F_2)$ from Proposition \ref{Sigma-1}, with $\Sigma_{-1}$, the blow-up of $\mathbb{P}^2$ at a point, it is clear that the open subvariety $\mathbb{P}^2 \backslash \{a\}$ corresponds to pairs $(F_1, F_2)$ with $F_1 \neq \mathbb{C}\{x^{n-2}v\}$. To finish the proof, note that the latter space is precisely $\overline{C_n^n}$, using Proposition \ref{cni}. It follows that $\overline{C_n^n} \backslash C_n^n \simeq \mathbb{P}^1$; this is also clear from the proof of Proposition \ref{P1bundle}. \end{proof}

\begin{example} 
The simplest example is $\overline{C_2^2}$. There is only one irreducible component. The space $(\mathbb{C}(v))^{\perp}$ has dimension 3, and $\mathbb{C}v \subset F_2 \subset (\mathbb{C}v)^{\perp}$. Sending a flag $F_{\bullet}$ to its $F_2$ gives a map $\overline{C_2^2} \rightarrow \mathbb{P}((\mathbb{C}v)^{\perp} / \mathbb{C}v)$, and hence $\overline{C_2^2}$ is a $\mathbb{P}^1$-bundle over $\mathbb{P}^1$.  \vspace{5pt}

The component $\overline{C_3^3}$ can be described similarly. The conditions are that $v\in F_3\setminus F_2$ and $xv\in F_2\setminus F_1$. We need to choose $F_1\in \mathbb{P}(\ker(x)\cap (\C v)^\perp)\setminus \C (xv)$ again this is a $\mathbb{P}^2$ with one point removed. Once $F_1$ is chosen, the conditions tell us that $F_2=F_1+\C \{xv\}$ and $F_3=F_2+\C\{v\}$ are determined.
\end{example}

\textbf{Case 2.}  For $1 \leq i < j \leq n$, let $T_{i,j}$ be the bitableau of shape $((n-2), (2))$ with the second tableau containing the numbers $\{ i, j \}$ in a single row, and the first tableau containing the other $n-2$ numbers in a single row. Let $C_n^{i,j} = \Phi^{-1}(T_{i,j})$, and let its closure $\overline{C_n^{i,j}}$ be the irreducible component of $\cC_{(v,x)}$ corresponding to $T_{i,j}$. 
\begin{proposition} \label{trivialP1bundle} 
Suppose $j-i>1$, then $\overline{C_n^{i,j}} \simeq \mathbb{P}^1 \times \mathbb{P}^1$. 
\end{proposition}
\begin{proof} 
Let $F_{\bullet} \in C_n^{i,j}$, with $j-i>1$ (this implies that $n\geq 3$). For $1\leq k\leq n-j$, we have $\eType(v+F_k, x|_{F_k^{\perp}/F_k})=((n-k-2),(2))$. For $k=1$, this implies that
$ (x|_{F_k^{\perp}/F_k})^{n-3}(v+F_1) = 0$ which means that $x^{n-3}v\in F_1$; but since $x^{n-3}v\neq 0$, we have that $F_1=\C\{x^{n-3}v\}$. For $1<k\leq n-j$, recursively, we obtain that $x^{n-k-2}v\in F_k$, hence
\begin{equation}\label{eq:ijFk} 
F_k=\C\{x^{n-3}v, \ldots, x^{n-k-2}v\}~\qquad 1\leq k\leq n-j.
\end{equation}

Now let $n-j<k<n-i$. We have $\eType(v+F_k, x|_{F_k^{\perp}/F_k})=((n-k-1),(1))$, hence $x^{n-k-2}v \notin F_k$. Also, $\eType(v+F_{k+1}, x|_{F_{k+1}^{\perp}/F_{k+1}})=((n-k-2),(1))$, hence $x^{n-k-2}v \in F_{k+1}$ and $F_{k+1} = F_k \oplus \mathbb{C}\{x^{n-k-2}v\}$. In particular, $F_{n-j+2} = F_{n-j+1} \oplus \mathbb{C}\{x^{j-3}v\} \subset x^{-1}(F_{n-j})$, but by counting dimensions we actually get an equality 
\begin{equation}\label{eq:Fnj2} F_{n-j+2}=x^{-1}(F_{n-j}).\end{equation}
We have also established that
\begin{equation}\label{eq:Fkij}F_{k} = F_{k-1} \oplus \mathbb{C}\{x^{n-k-1}v\}~\quad \text{for} \quad n-j+3\leq k\leq n-i.
\end{equation}

Suppose now that $k > n-i$. Then $\eType(v+F_k, x|_{F_k^{\perp}/F_k})=((n-k),\emptyset)$, hence $x^{n-k-1}v \notin F_k$. Also  $\eType(v+F_{k+1}, x|_{F_{k+1}^{\perp}/F_{k+1}})=((n-k-1),0)$, hence $x^{n-k-1}v \in F_{k+1}$ and $F_{k+1} = F_k \oplus \mathbb{C}x^{n-k-1}v$. In particular we get that $F_{n-i+2}=F_{n-i+1}\oplus Cx^{i-2}v=x^{-1}(F_{n-i})$. Thus: 
\begin{equation}
F_{n-i+2}=x^{-1}(F_{n-i}), \quad \text{and} \label{eq:Fi2}
\end{equation}
\begin{equation}
F_{k}= F_{k-1}\oplus \C x^{n-k}v~\quad \text{for} \quad n-i+3\leq k\leq n. \label{eq:Fk3}
\end{equation}
The conditions \eqref{eq:ijFk}-\eqref{eq:Fk3} are closed, hence they hold also in $\overline{C_n^{i,j}}$ and in fact they define it. Therefore all spaces in a flag $F_{\bullet}\in\overline{C_n^{i,j}}$ are determined uniquely except for $F_{n-j+1}\in \mathbb{P}(x^{-1}(F_{n-j})/F_{n-j})$ and $F_{n-i+1}\in \mathbb{P}(x^{-1}(F_{n-i})/F_{n-i})$, and it follows that $\overline{C_n^{i,j}} \simeq \mathbb{P}^1 \times \mathbb{P}^1$. \end{proof}

\begin{proposition} \label{Sigma-2} 
For $1 \leq i \leq n-1$, we have $\overline{C_n^{i,i+1}} \simeq \Sigma_{-2}$.
\end{proposition} 
\begin{proof} 
Let $F_\bullet\in C_n^{i,i+1}$ then, by the same argument as in Proposition \ref{trivialP1bundle}, we have the equation \eqref{eq:ijFk} for $1\leq k\leq n-i-1$. \vspace{5pt}

Since, for $n-i+1\leq k\leq n$, $\eType\left( v+F_{k},x|_{F_{k}^\perp/F_{k}}\right)=((n-k),\emptyset)$
we have that $x^{n-k-1}\in F_{k+1}\setminus F_k$, so
\begin{equation}\label{eq:Fk4}
F_{k+1} = F_k \oplus \mathbb{C}\{x^{n-k-1}v\}~\qquad \text{for} \quad n-i+1\leq k\leq n.
\end{equation}
Again, \eqref{eq:ijFk} and \eqref{eq:Fk4} are closed conditions, so they hold in $\overline{C_n^{i,i+1}}$. This means that any flag is determined by the choice of $F_{n-i}\in \mathbb{P}(x^{-1}(F_{n-i-1})/F_{n-i-1})$ and $F_{n-i+1}\in \mathbb{P}(x^{-1}(F_{n-i})/F_{n-i})$. \vspace{5pt}

In particular, there is an isomorphism $\overline{C_n^{i,i+1}} \simeq \overline{C_{i+1}^{i,i+1}}$. Hence it suffices to show the result in the case where $i=n-1$. We have established the following description: 
$$
\overline{C_n^{n-1,n}} \simeq \{ (F_1, F_2) \; | \; F_1 \in \mathbb{P}(\text{ker}(x)), F_1 \subset F_2 \subset x^{-1}(F_1) \}.
$$
We have a $\mathbb{P}^1$-bundle map, $\overline{C_n^{n-1,n}} \rightarrow \mathbb{P}(\text{ker}(x))$, and the fibre over $F_1 \subseteq \text{ker}(x)$ is $\mathbb{P}(x^{-1}(F_1)/F_1)$. As vector spaces, $x^{-1}(F_1)/F_1 \simeq x^{-1}(F_1)/\text{ker}(x) \oplus \text{ker}(x)/F_1$. Clearly the line bundle over $\mathbb{P}(\text{ker}(x))$ with fibre over $F_1$ being $\text{ker}(x)/F_1$ is isomorphic to $\mathcal{O}(1)$. Writing out a basis for $x^{-1}(F_1)$, we see that the line bundle over $\mathbb{P}(\text{ker}(x))$ with fibre over $F_1$ being $x^{-1}(F_1)/\text{ker}(x)$ is isomorphic to the one whose fibre over $F_1$ is $F_1$ itself; and hence can be identified with $\mathcal{O}(-1)$. Hence $\overline{C_n^{n-1,n}} \simeq \mathbb{P}(\mathcal{O}(1) \oplus \mathcal{O}(-1)) \simeq \Sigma_{-2}$, as required (see Remark \ref{tensorL}).  
\end{proof}

\begin{corollary} Suppose $j-i>1$ and $l-k>1$. Then the intersection $\overline{C_n^{i,j}} \cap \overline{C_n^{k,l}}$ is empty unless $|i-k| \leq 1$ and $|j-l| \leq 1$. Furthermore,  $\overline{C_n^{i,j}} \cap \overline{C_n^{i,j+1}} \simeq \mathbb{P}^1 \simeq \overline{C_n^{i-1,j}} \cap \overline{C_n^{i,j}}$, while both intersections $\overline{C_n^{i,j}} \cap \overline{C_n^{i+1,j+1}}$ and $\overline{C_n^{i,j}} \cap \overline{C_n^{i-1,j+1}}$ consist of a single point. \end{corollary}
\begin{proof} Suppose $j-i>1$, $l-k>1$, and $j \geq l+2$. Then if $F_{\bullet} \in \overline{C_n^{i,j}}$, then from the proof of Proposition \ref{trivialP1bundle}, $F_{n-j} = \mathbb{C}\{ x^{n-3}v, \ldots, x^{j-2}v \}$ and $F_{n-j+2} = x^{-1}(F_{n-j})$. If $F_{\bullet} \in \overline{C_n^{k,l}}$, again from the proof of Proposition \ref{trivialP1bundle}, for $r \geq l$, $F_{n-r} = \mathbb{C}\{ x^{n-3}v, \ldots, x^{r-2}v \}$; in particular, $F_{n-j+2} = \mathbb{C}\{ x^{n-3}v, \ldots, x^{j-4}v \} \neq x^{-1}(F_{n-j})$. Hence, in this case, $\overline{C_n^{i,j}} \cap \overline{C_n^{k,l}} = \emptyset$.  \vspace{5pt}

Next suppose $j-i>1$ and $l-k>1$ and $i \geq k+2$. Then if $F_{\bullet} \in \overline{C_n^{i,j}}$, then from the proof of Proposition \ref{trivialP1bundle}, $F_{n-i+1} = \text{ker}(x^2) \oplus \mathbb{C}\{x^{n-5}v, \cdots, x^{i-1}v \}$. However if $F_{\bullet} \in \overline{C_n^{k,l}}$, then from Proposition \ref{trivialP1bundle} again, since $i \geq k+2$ we have $F_{n-i+1} = \text{ker}(x) \oplus \mathbb{C}\{ x^{n-4}v, \cdots, x^{i-2}v \}$. In this case, again we have that $\overline{C_n^{i,j}} \cap \overline{C_n^{k,l}} = \emptyset$. \vspace{5pt}

Next suppose $j-i > 1$, and $F_{\bullet} \in \overline{C_n^{i,j}} \cap \overline{C_n^{i,j+1}}$; then 
$$ 
\mathbb{C}\{ x^{n-3}v, \cdots, x^{i-1}v \} \subset F_{n-i} \subset \text{ker}(x) \oplus \mathbb{C}\{ x^{n-4}v, \cdots, x^{i-2}v \}. 
$$ All other vector spaces are determined uniquely, and it follows that $\overline{C_n^{i,j}} \cap \overline{C_n^{i,j+1}} \simeq \mathbb{P}^1$; similarly $\overline{C_n^{i-1,j}} \cap \overline{C_n^{i,j}} \simeq \mathbb{P}^1$. Using the same method, one verifies that both intersections $\overline{C_n^{i,j}} \cap \overline{C_n^{i+1,j+1}}$ and $\overline{C_n^{i,j}} \cap \overline{C_n^{i-1,j+1}}$ consist of a single point. \end{proof}

\section{Further directions}

\subsection{Canonical bases in irreducible Weyl group representations} Given a bipartition $(\mu, \nu) \in \mathcal{Q}_n$, the top homology group $H_{\text{top}}(\mathcal{C}_{(v,x)})$ carries an action of the Weyl group $W(C_n) = \Z_2\wr S_n$, and realises the irreducible representation indexed by that bipartition. This follows from Kato's results in \cite{kat} and \cite{kato2}. Now the classes of the irreducible components give a distinguished basis in the Weyl group representation. It would be interesting to study this basis, and determine whether it has properties that resemble the canonical bases introduced by Kazhdan and Lusztig in various contexts. In type A, a version of this question was first posed in \cite{KL80} and subsequently restated in terms of the irreducibility of certain characteristic cycles by Kashiwara-Tanisaki (\cite{KT84}). Subsequent work by Kashiwara-Saito (\cite{KS97}) and Williamson (\cite{Wil15}) found counterexamples to the irreducibility conjectures. In the papers mentioned above, Kato also shows that the total homology $H_{\bullet}(\mathcal{C}_{(v,x)})$ realises the standard modules for certain multi-parameter Hecke algebras. It would be interesting to see if one can deduce any additional information about these modules using our results about the structure of exotic Springer fibres. This question is not well-understood in type A, and answering it would require new techniques. 


\def\cprime{$'$} \newcommand{\arxiv}[1]{\href{http://arxiv.org/abs/#1}{\tt
  arXiv:\nolinkurl{#1}}}

\end{document}